\pdfoutput=1
\RequirePackage{ifpdf}
\ifpdf % We~are running pdfTeX in pdf mode
\documentclass[pdftex]{sigma}
\else
\documentclass{sigma}
\fi

\numberwithin{equation}{section}

\newtheorem{Theorem}{Theorem}[section]
\newtheorem*{Theorem*}{Theorem}
\newtheorem{Corollary}[Theorem]{Corollary}

\newtheorem{Proposition}[Theorem]{Proposition}
 { \theoremstyle{definition}

 }

\begin{document}

\allowdisplaybreaks

\newcommand{\arXivNumber}{2309.09341}

\renewcommand{\PaperNumber}{083}

\FirstPageHeading

\ShortArticleName{Kernel Function, $q$-Integral Transformation and $q$-Heun Equations}

\ArticleName{Kernel Function, $\boldsymbol{q}$-Integral Transformation\\
 and $\boldsymbol{q}$-Heun Equations}

\Author{Kouichi TAKEMURA}

\AuthorNameForHeading{K.~Takemura}

\Address{Department of Mathematics, Ochanomizu University,\\ 2-1-1 Otsuka, Bunkyo-ku, Tokyo 112-8610, Japan}
\Email{\href{mailto:takemura.kouichi@ocha.ac.jp}{takemura.kouichi@ocha.ac.jp}}

\ArticleDates{Received February 14, 2024, in final form September 12, 2024; Published online September 19, 2024}

\Abstract{We find kernel functions of the $q$-Heun equation and its variants. We apply them to obtain $q$-integral transformations of solutions to the $q$-Heun equation and its variants. We discuss special solutions of the $q$-Heun equation from the perspective of the $q$-integral transformation.}


\Classification{39A13; 44A20}

\section{Introduction}\label{section1}

Heun's differential equation is given by
\begin{gather}
\frac{{\rm d}^2y}{{\rm d}z^2} + \left( \frac{\gamma}{z}+\frac{\delta }{z-1}+\frac{\epsilon}{z-t}\right) \frac{{\rm d}y}{{\rm d}z} + \frac{\alpha \beta z -B}{z(z - 1)(z - t)} y = 0,
\label{eq:Heun}
\end{gather}
with the condition $\gamma +\delta +\epsilon =\alpha +\beta +1$.
It sometimes appears in the study of mathematical physics.
Relationship with the Kerr black holes is remarkable among them.
The parameter $B$ in equation \eqref{eq:Heun} is called the accessory parameter, and it plays special roles in the analysis of the Heun equation.

It is known that Heun's differential equation admits integral transformations.
Kazakov and Slavyanov established Euler integral transformations of Heun's differential equation in \cite{KS}.
Set%
\begin{gather*}
 \gamma '=\gamma +1 -\alpha , \qquad \delta' =\delta +1-\alpha , \qquad \epsilon '=\epsilon +1-\alpha , \qquad \alpha '=2-\alpha , \\
\beta '= -\alpha +\beta +1 ,\qquad B'=B+(1-\alpha )(\epsilon +\delta t+(\gamma -\alpha ) (t+1))
\end{gather*}
and assume that the function $v(w)$ is a solution to the Heun equation with the parameters $\gamma '$, $\delta' $, $\epsilon '$, $\alpha '$, $\beta '$, $B'$, i.e.,
\begin{gather}
\frac{{\rm d}^2 v}{{\rm d}w^2} + \left( \frac{\gamma '}{w}+\frac{\delta '}{w-1}+\frac{\epsilon '}{w-t}\right) \frac{{\rm d}v}{{\rm d}w} + \frac{\alpha ' \beta ' w -B '}{w(w - 1)(w - t)} v = 0,
\label{eq:Heunw}
\end{gather}
then it was established in \cite{KS} that the function
\begin{equation*}
y(z)=\int _{C} v(w) (z-w)^{-\alpha } {\rm d}w
\end{equation*}
is a solution to equation \eqref{eq:Heun} for a suitable cycle $C$.
Examples of the cycles were described explicitly in \cite{Kaz1,KS}.
Note that the integral transformation was also obtained in \cite{TakI,TakM} by considering the middle convolution for some special system of Fuchsian equations.
As an application of the integral transformation of the Heun equation, a correspondence of special solutions of the Heun equation was obtained in \cite{TakIT}.
Namely, the polynomial-type solutions (i.e., the solutions of equation \eqref{eq:Heunw} written as a product of a polynomial and a function $w^{\rho _0} (w-1)^{\rho _1} (w-t)^{\rho _t}$ for some constants $\rho _0$, $\rho _1$, $\rho _t$) correspond to the solutions written as a finite sum of the hypergeometric functions by the integral transformation.

Some of integral transformations are related with the kernel function.
Let ${\bf x} =(x_1, \dots, x_m)$ and ${\bf y} =(y_1, \dots, y_n)$ be the variables and $({\mathcal A}_{\bf x}, {\mathcal B}_{\bf y})$ be a pair of operators which act on meromorphic functions in {\bf x} and {\bf y} respectively.
In \cite{KNS}, $\Phi ({\bf x}; {\bf y})$ is called a kernel function for the pair $({\mathcal A}_{\bf x}, {\mathcal B}_{\bf y})$, if it satisfies a functional equation of the form
$
{\mathcal A}_{\bf x} \Phi ({\bf x}; {\bf y}) = {\mathcal B}_{\bf y} \Phi ({\bf x}; {\bf y})$.
Langmann studied the kernel functions systematically in the analysis of eigenfunctions for quantum integrable systems such as the Calogero--Moser--Sutherland system and the Inozemtsev system~\cite{L10,LT}.
Note that we can regard the eigenvalue problem of the Inozemtsev system as a~multi-variable generalization of Heun's differential equation (see \cite{TakS} for a review).
By using the kernel function identity for the Inozemtsev system, we can obtain integral transformations of eigenfunctions of the Inozemtsev system and related systems \cite{LT}.
The Ruijsenaars--van Diejen system (see \cite{RuiN,vD0}) is a relativistic (or discrete) quantum integrable system of the Inozemtsev type, and kernel functions of them were studied by Ruijsenaars \cite{Rui06,Rui09} and Komori, Noumi and Shiraishi \cite{KNS}.
Atai and Noumi \cite{AN} applied the kernel functions of the Ruijsenaars--van Diejen system to the integral transformations of the eigenfunctions.

In this paper, we obtain kernel function identities for $q$-deformations of Heun's differential equations and apply them to integral transformations.
Here, the $q$-Heun equation was introduced by Hahn \cite{Hahn} as the form
\begin{align}
 \big\{ a_2 x^2 +a_1 x+ a_0 \big\} g(x/q) -\big\{ b_2 x^2 + b_1 x + b_0 \big\} g(x) + \big\{ c_2 x^2 + c_1 x+ c_0 \big\} g(xq) =0,
 \label{eq:qHeun}
\end{align}
with the condition $a_2 a_0 c_2 c_0 \neq 0$.
Degenerations of the Ruijsenaars--van Diejen system were investigated in \cite{TakR,vD0} and the $q$-Heun equation was rediscovered in \cite{TakR} by considering degeneration of the Ruijsenaars--van Diejen system four times.
It was obtained as an eigenvalue equation of the fourth degeneration of the Ruijsenaars--van Diejen operator of one variable
\begin{gather}
 A^{\langle 4 \rangle} ( x; h_1, h_2, l_1, l_2, \alpha _1, \alpha _2, \beta ) \nonumber \\
\qquad = x^{-1} \big(x-q^{h_1 + 1/2} t_1\big) \big(x-q^{h_2 +1/2} t_2\big) T_{x}^{-1} \nonumber \\
 \phantom{\qquad =}{}+ q^{\alpha _1 +\alpha _2} x^{-1} \big(x - q^{l_1 -1/2} t_1\big) \big(x - q^{l_2 -1/2} t_2\big) T_{x} \nonumber\\
 \phantom{\qquad =}{}-\big\{ \big(q^{\alpha _1} +q^{\alpha _2} \big) x + q^{(h_1 +h_2 + l_1 +l_2 +\alpha _1 +\alpha _2 )/2}\big( q^{\beta/2} + q^{-\beta/2} \big) t_1 t_2 x^{-1} \big\}, \label{eq:A4op0}
\end{gather}
where $T_{x} ^{\pm 1} g(x)=g\big(q ^{\pm 1} x\big) $.
Namely, equation \eqref{eq:qHeun} admits an expression as
\begin{equation*}
A^{\langle 4 \rangle} ( x; h_1, h_2, l_1, l_2, \alpha _1, \alpha _2, \beta ) g(x) =Eg(x),
\end{equation*}
where $E$ is an arbitrary complex number.
The eigenvalue $E$ corresponds to the parameter $b_1 $ in equation \eqref{eq:qHeun}, and it plays a role of the accessory parameter.
As $q\to 1$, we essentially obtain the differential equation
\begin{gather}
\frac{{\rm d}^2y}{{\rm d}z^2} + \left( \frac{\gamma}{z}+\frac{\delta }{z-t_1}+\frac{\epsilon}{z-t_2}\right) \frac{{\rm d}y}{{\rm d}z} + \frac{\alpha \beta z -B}{z(z - t_1)(z - t_2)} y = 0 ,
 \label{eq:Heunt1t2}
\end{gather}
which is equivalent to Heun's differential equation.
Note that the parameters $t_1$ and $t_2$ are common in equations \eqref{eq:A4op0} and \eqref{eq:Heunt1t2}.
Other $q$-deformations of Heun's differential equations were introduced in \cite{TakqH} by focusing on the third degenerate Ruijsenaars--van Diejen operator~$A^{\langle 3 \rangle} ( x; h_1, h_2, h_3, l_1, l_2, l_3, \alpha , \beta )$ (see equation \eqref{eq:qthirdal}) and the second degenerate one $A^{\langle 2 \rangle} ( x; h_1, \allowbreak h_2, h_3, h_4, l_1, l_2, l_3, l_4 , \alpha )$ (see equation \eqref{eq:qsecondal}), and they were called the variants of the $q$-Heun equation.

In this paper, we obtain kernel function identities for the $q$-Heun equation and its variants.
Set
\begin{equation*}
P^{(1)}_{\mu ,\mu _0 } (x,s)= \frac{(q^{\mu }s/x;q)_{\infty }}{(q^{\mu _0 }s/x;q)_{\infty }}, \qquad (a;q)_{\infty } = \prod _{j=0}^{\infty } \big(1-q ^j a\big).
\end{equation*}
We obtain a kernel function identity for the $q$-Heun equation as follows.

\begin{Theorem} \label{thm:A4KerId0}
If the parameters satisfy
\begin{gather*}
 \chi = \big( \tilde{h}_1 + \tilde{h}_2 - \tilde{l}_1 - \tilde{l}_2 + \tilde{\alpha }_1 - \tilde{\alpha }_2 - \tilde{\beta } \big)/2 , \qquad \nu = \mu _0 + \alpha _1 - \tilde{\alpha }_2 , \qquad \mu = \mu _0 + \chi + 1, \\
 \beta = - \tilde{\beta } - \chi , \qquad \alpha _2 = \alpha _1 + \tilde{\alpha }_1 - \tilde{\alpha }_2 - \chi , \\
  l_i =\tilde{h}_i + \mu _0 , \qquad h_i = \tilde{l}_i +\mu _0 + \chi ,\qquad i=1,2,
\end{gather*}
then the function \smash{$\Phi (x,s ) = x^{- \alpha _1} s^{1+\chi - \tilde{\alpha }_1} P^{(1)}_{\mu ,\mu _0} (x,s)$} satisfies
\begin{equation*}
 A^{\langle 4 \rangle} ( x; h_1, h_2, l_1, l_2, \alpha _1, \alpha _2, \beta ) \Phi (x,s ) = q^{\nu } A^{\langle 4 \rangle} \big( s; \tilde{h}_1, \tilde{h}_2, \tilde{l}_1, \tilde{l}_2, \tilde{\alpha }_1, \tilde{\alpha }_2, \tilde{\beta } \big) \Phi (x,s ) .
\end{equation*}
\end{Theorem}
We also obtain kernel function identities for the variants of the $q$-Heun equation in Theorems~\ref{thm:A3KerId} and \ref{thm:A2KerId}.
As an application, we have $q$-integral transformations of solutions to the $q$-Heun equation and its variants.
A $q$-integral transformation for the $q$-Heun equation was already obtained in \cite{STT2} by using the $q$-middle convolution which was introduced by Sakai and Yamaguchi \cite{SY}.
However, convergence of the $q$-integral transformation associated with the $q$-middle convolution was not discussed in \cite{SY}.
In this paper, we obtain $q$-integral transformations for the $q$-Heun equation and its variants without using the $q$-middle convolution and we discuss the convergence directly.
It seems that $q$-integral transformations for the variants of the $q$-Heun equation have not been known before.

This paper is organized as follows.
In Section~\ref{sec:Kerfunc}, we obtain identities for kernel functions related to the $q$-Heun equation and its variants.
In Section~\ref{sec:qinttrans}, we formulate a method of finding a $q$-integral transformation by using a kernel function and its identity.
In Section~\ref{sec:qintHeunvar}, we obtain $q$-integral transformations for the $q$-Heun equation and its variants.
In Section~\ref{sec:solqHeun}, we discuss special solutions of the $q$-Heun equation from the perspective of the $q$-integral transformation.
In Section~\ref{sec:CR}, we give concluding remarks.

Throughout this paper, we assume that $q$ is a complex number such that $0<|q|<1$.

\section[Kernel functions related to q-Heun equations]{Kernel functions related to $\boldsymbol{q}$-Heun equations} \label{sec:Kerfunc}

We introduce functions which appear in the kernel function identities.
Recall that the function~\smash{$P^{(1)}_{\mu ,\mu _0 } (x,s) $} was defined as
\begin{equation*}
P^{(1)}_{\mu ,\mu _0 } (x,s)= \frac{(q^{\mu }s/x;q)_{\infty }}{(q^{\mu _0 }s/x;q)_{\infty }}.
\end{equation*}
It satisfies
\begin{align}
& P_{\mu ,\mu _0} (x,s/q)= P_{\mu ,\mu _0 } (qx,s) = \frac{x-q^{\mu -1 }s}{x-q^{\mu _0 -1} s} P_{\mu ,\mu _0 } (x,s), \nonumber \\
& P_{\mu ,\mu _0 } (x,qs)= P_{\mu ,\mu _0 } (x/q,s) = \frac{x-q^{\mu _0 } s}{x-q^{\mu }s} P_{\mu ,\mu _0 } (x,s) . \label{eq:Pmumuprel}
\end{align}
The function
\begin{equation*}
P^{(2)}_{\mu ,\mu _0 } (x,s)= (x/s)^{\mu - \mu _0 } \frac{\big(q^{- \mu _0 +1 } x/s;q\big)_{\infty }}{\big(q^{- \mu +1 } x/s ;q\big)_{\infty }}
\end{equation*}
also satisfies equation \eqref{eq:Pmumuprel}.
Set $\vartheta _q (t) = (t, q/t, q ;q)_{\infty } ( = (t ; q)_{\infty } (q/t ; q)_{\infty } (q ;q)_{\infty } ) $.
Then we have
\begin{gather}
\frac{P^{(2)}_{\mu ,\mu _0 } (x,q^n \xi )}{P^{(1)}_{\mu ,\mu _0 } (x,q^n \xi )} = \frac{P^{(2)}_{\mu ,\mu _0 } (x,\xi )}{P^{(1)}_{\mu ,\mu _0 } (x,\xi )} = (x/\xi )^{\mu - \mu _0 } \frac{\vartheta _q \big(q^{- \mu _0 +1 } x/\xi \big)}{\vartheta _q \big(q^{- \mu +1 } x/\xi \big)} \label{eq:P2P1theta}
\end{gather}
for $n \in {\mathbb Z} $.
On the limit with respect to the variable $s$, we have
\begin{gather}
\lim _{s \to 0} P^{(1)}_{\mu ,\mu _0 } (x,s) = 1, \qquad \lim _{s \to \infty } s^{\mu - \mu _0 } P^{(2)}_{\mu ,\mu _0 } (x,s) = x^{\mu - \mu _0 } ,
\label{eq:P12lim1}
\end{gather}
and it follows from equation \eqref{eq:P2P1theta} that
\begin{gather*}
 \lim _{L \to + \infty } P^{(2)}_{\mu ,\mu _0 } (x,s) | _{s=q^{L} \xi } = (x/\xi )^{\mu - \mu _0 } \frac{\vartheta _q \big(q^{- \mu _0 +1 } x/\xi \big)}{\vartheta _q \big(q^{- \mu +1 } x/ \xi \big)} , \\
 \lim _{K \to - \infty } s^{\mu - \mu _0 } P^{(1)}_{\mu ,\mu _0 } (x,s) | _{s=q^{K-1 } \xi } = \xi ^{\mu - \mu _0} \frac{\vartheta _q \big(q^{- \mu +1 } x/ \xi \big)}{\vartheta _q \big(q^{- \mu _0 +1 } x/\xi \big)} .
\end{gather*}

Recall that the fourth degeneration of the Ruijsenaars--van Diejen operator of one variable is written as
\begin{gather}
 A^{\langle 4 \rangle} ( x; h_1, h_2, l_1, l_2, \alpha _1, \alpha _2, \beta ) \nonumber \\
 \qquad = x^{-1} \big(x-q^{h_1 + 1/2} t_1\big) \big(x-q^{h_2 +1/2} t_2\big) T_{x}^{-1} \nonumber \\
\phantom{\qquad =}{}+ q^{\alpha _1 +\alpha _2} x^{-1} \big(x - q^{l_1 -1/2} t_1\big) \big(x - q^{l_2 -1/2} t_2\big) T_{x}\nonumber\\
\phantom{\qquad =}{} -\big\{ \big(q^{\alpha _1} +q^{\alpha _2} \big) x + q^{(h_1 +h_2 + l_1 +l_2 +\alpha _1 +\alpha _2 )/2}\big( q^{\beta/2} + q^{-\beta/2} \big) t_1 t_2 x^{-1} \big\}, \label{eq:A4op}
\end{gather}
where $T_{x} ^{\pm 1} g(x)=g\big(q ^{\pm 1} x\big) $, and it is related with the $q$-Heun equation as explained in Section~\ref{section1}.
We obtain a kernel function and its identity for the operator $ A^{\langle 4 \rangle} ( x; h_1, h_2, l_1, l_2, \alpha _1, \alpha _2, \beta ) $ as follows.
\begin{Theorem} \label{thm:A4KerId}
If the parameters satisfy
\begin{gather}
 \chi = \big( \tilde{h}_1 + \tilde{h}_2 - \tilde{l}_1 - \tilde{l}_2 + \tilde{\alpha }_1 - \tilde{\alpha }_2 - \tilde{\beta } \big)/2 , \qquad \nu = \mu _0 + \alpha _1 - \tilde{\alpha }_2 ,\qquad \mu = \mu _0 + \chi + 1, \nonumber \\
 \beta = - \tilde{\beta } - \chi , \qquad \alpha _2 = \alpha _1 + \tilde{\alpha }_1 - \tilde{\alpha }_2 - \chi , \nonumber\\
 l_i =\tilde{h}_i + \mu _0 , \qquad h_i = \tilde{l}_i +\mu _0 + \chi , \qquad i=1,2, \label{eq:A4KerIdp}
\end{gather}
then we have
\begin{align}
& A^{\langle 4 \rangle} ( x; h_1, h_2, l_1, l_2, \alpha _1, \alpha _2, \beta ) x^{- \alpha _1} s^{1+\chi - \tilde{\alpha }_1} P_{\mu ,\mu _0} (x,s) \nonumber \\
&\qquad = q^{\nu } A^{\langle 4 \rangle} \big( s; \tilde{h}_1, \tilde{h}_2, \tilde{l}_1, \tilde{l}_2, \tilde{\alpha }_1, \tilde{\alpha }_2, \tilde{\beta } \big) x^{- \alpha _1} s^{1+\chi - \tilde{\alpha }_1} P_{\mu ,\mu _0 } (x,s) , \label{eq:A4KerId}
\end{align}
where $P_{\mu ,\mu _0 } (x,s) $ is a function which satisfies equation \eqref{eq:Pmumuprel}.
\end{Theorem}

\begin{proof}
Set
\begin{gather*}
 a_0 (x) = x^{-1} \big(x-q^{h_1 + 1/2} t_1\big) \big(x-q^{h_2 +1/2} t_2\big) , \\
 b_0 (x) = - \big(q^{\alpha _1} +q^{\alpha _2} \big) x - q^{(h_1 +h_2 + l_1 +l_2 +\alpha _1 +\alpha _2 )/2} \big( q^{\beta/2} + q^{-\beta/2} \big) t_1 t_2 x^{-1} , \\
 c_0 (x) = q^{\alpha _1 +\alpha _2} x^{-1} \big(x - q^{l_1 -1/2} t_1\big) \big(x - q^{l_2 -1/2} t_2\big) , \\
 \tilde{a}_0 (s) = s^{-1} \big(s-q^{\tilde{h}_1 + 1/2} t_1\big) \big(s-q^{\tilde{h}_2 +1/2} t_2\big) , \\
 \tilde{b}_0 (s) = - \big(q^{\tilde{\alpha }_1} +q^{\tilde{\alpha }_2} \big) s - q^{(\tilde{h}_1 +\tilde{h}_2 + \tilde{l}_1 +\tilde{l}_2 +\tilde{\alpha }_1 +\tilde{\alpha }_2 )/2} \big( q^{\tilde{\beta }/2} + q^{-\tilde{\beta }/2} \big) t_1 t_2 s^{-1} , \\
 \tilde{c}_0 (s) = q^{\tilde{\alpha }_1 +\tilde{\alpha }_2} s^{-1} \big(s - q^{\tilde{l}_1 -1/2} t_1\big) \big(s - q^{\tilde{l}_2 -1/2} t_2\big) .
\end{gather*}
It follows from equation \eqref{eq:Pmumuprel} that equation \eqref{eq:A4KerId} is equivalent to the equation
\begin{gather}
 \big( q^{ \alpha _1} a_0 (x) - q^{\nu + 1+\chi - \tilde{\alpha }_1}\tilde{c}_0 (s) \big) \frac{x-q^{\mu _0 } s}{x-q^{\mu }s} \nonumber \\
\qquad {}+ \big( q^{- \alpha _1} c_0 (x) - q^{\nu -1-\chi + \tilde{\alpha }_1} \tilde{a}_0 (s) \big) \frac{x-q^{\mu -1 }s}{x-q^{\mu _0 -1} s} + b_0 (x) - q^{\nu } \tilde{b}_0 (s) =0 . \label{eq:A4KerIdequiv}
\end{gather}
By using the relations in equation \eqref{eq:A4KerIdp}, we have
\begin{gather*}
 q^{- \alpha _1} c_0 (x) - q^{\nu -1-\chi + \tilde{\alpha }_1} \tilde{a}_0 (s) \\
\qquad = q^{\alpha _2} \big\{ x^{-1} \big(x - q^{l_1 -1/2} t_1\big) \big(x - q^{l_2 -1/2} t_2\big)\\
\phantom{\qquad = }{} - q^{\mu _0 -1} s^{-1} \big(s-q^{l_1 -\mu _0 + 1/2} t_1\big) \big(s-q^{l_2-\mu _0 +1/2} t_2\big) \big\} \\
 \qquad= q^{\alpha _2} \big( x - q^{\mu _0 -1} s \big) \big( 1 - q^{l_1+l_2 -\mu _0 } t_1 t_2 x^{-1} s^{-1} \big) , \\
 q^{ \alpha _1} a_0 (x) - q^{\nu + 1+\chi - \tilde{\alpha }_1}\tilde{c}_0 (s) \\
 \qquad= q^{ \alpha _1} \big\{ x^{-1} \big(x-q^{h_1 + 1/2} t_1\big) \big(x-q^{h_2 +1/2} t_2\big) \\
 \phantom{\qquad = }{} - q^{\mu } s^{-1} \big(s - q^{h_1 +1/2 - \mu } t_1\big) \big(s - q^{h_2 +1/2 - \mu } t_2\big) \big\} \\
 \qquad= q^{ \alpha _1} \big( x - q^{\mu } s \big)\big( 1 - q^{h_1 + h_2 +1- \mu } t_1 t_2 x^{-1} s ^{-1} \big) .
\end{gather*}
Hence
\begin{gather*}
 \big( q^{ \alpha _1} a_0 (x) - q^{\nu + 1+\chi - \tilde{\alpha }_1}\tilde{c}_0 (s) \big) \frac{x-q^{\mu _0 } s}{x-q^{\mu }s} + \big( q^{- \alpha _1} c_0 (x) - q^{\nu -1-\chi + \tilde{\alpha }_1} \tilde{a}_0 (s) \big) \frac{x-q^{\mu -1 }s}{x-q^{\mu _0 -1} s} \\
 \qquad= q^{ \alpha _1} \big(x-q^{\mu _0 } s \big)\big( 1 - q^{h_1 + h_2 +1- \mu } t_1 t_2 x^{-1} s ^{-1} \big) \\
 \phantom{\qquad=}{} + q^{\alpha _2} \big(x-q^{\mu -1 }s \big) \big( 1 - q^{l_1+l_2 -\mu _0 } t_1 t_2 x^{-1} s^{-1} \big) \\
\qquad = \big( q^{ \alpha _1} + q^{\alpha _2} \big) x + \big(q^{\alpha _1+ h_1 + h_2 +1 + \mu _0 - \mu } + q^{\alpha _2 + l_1+l_2 -\mu _0 + \mu -1 } \big) t_1 t_2 x^{-1} \\
 \phantom{\qquad=}{}-\big( q^{\alpha _1 +\mu _0 } + q^{\alpha _2 + \mu -1 }\big)s - \big(q^{\alpha _1 + h_1 + h_2 +1- \mu } + q^{\alpha _2 + l_1+l_2 -\mu _0 } \big) t_1 t_2 s^{-1} ,
\end{gather*}
and we obtain equation \eqref{eq:A4KerIdequiv} by using equation \eqref{eq:A4KerIdp}.
\end{proof}

Note that equation \eqref{eq:A4KerIdp} is equivalent to
\begin{gather*}
 \chi = ( h_1 + h_2 - l_1 - l_2 + \alpha _1 - \alpha _2 - \beta )/2 , \qquad \nu = \mu _0 + \alpha _1 - \tilde{\alpha }_2 ,\qquad\mu = \mu _0 + \chi + 1, \\
 \tilde{\beta } = - \beta - \chi , \qquad \tilde{\alpha }_1 = \tilde{\alpha }_2 + \alpha _2 - \alpha _1 + \chi , \\
  \tilde{h}_i = l_i - \mu _0 , \qquad \tilde{l}_i = h_i - \mu _0 - \chi ,\qquad
 i=1,2.
\end{gather*}
Theorem \ref{thm:A4KerId0} follows immediately from Theorem \ref{thm:A4KerId}, because the function~\smash{$P^{(1)}_{\mu ,\mu _0 } (x,s) $} satisfies equation \eqref{eq:Pmumuprel}.

In order to define the variants of the $q$-Heun equation, we introduce the following operators:
\begin{gather}
 A^{\langle 3 \rangle} ( x; h_1, h_2, h_3, l_1, l_2, l_3, \alpha , \beta ) \nonumber \\
\qquad = x^{-1} \prod_{n=1}^3 \big(x- q^{h_n+1/2} t_n\big) T_{x}^{-1} + q^{2\alpha +1} x^{-1} \prod_{n=1}^3 \big(x- q^{l_n -1/2} t_n \big) T_x \nonumber \\
 \phantom{\qquad = }{} + q^{\alpha +1/2} \Biggl[ -\bigl(q^{1/2} +q^{-1/2} \bigr) x^2 +\sum _{n=1}^3 \big( q^{h_n} + q^{l_n} \big)t_n x  \nonumber \\
 \phantom{\qquad = }{}  + q^{(l_1 +l_2 +l_3 +h_1 +h_2 +h_3)/2} \big( q^{\beta/2} + q^{-\beta/2} \big) t_1 t_2 t_3 x^{-1} \Biggr] , \label{eq:qthirdal}
\\
 A^{\langle 2 \rangle} ( x; h_1, h_2, h_3, h_4, l_1, l_2, l_3, l_4 , \alpha ) \nonumber \\
 \qquad= x^{-2} \prod_{n=1}^4 \big(x- q^{h_n +1/2} t_n\big) T_{x}^{-1} + q^{2\alpha +1} x^{-2} \prod_{n=1}^4 \big(x- q^{l_n -1/2} t_n\big) T_{x} \nonumber \\
 \phantom{\qquad = }{} + q^{\alpha +1/2} \Biggl[ -\big(q^{1/2} +q^{-1/2} \big) x^2 + \sum _{n=1}^4 \big(q^{h_n}+ q^{l_n} \big) t_n x  \nonumber \\
  \phantom{\qquad = }{} + \prod_{n=1}^4 q^{(h_n +l_n)/2} t_n \left\{ - \big(q^{1/2} +q^{-1/2} \big) x^{-2} + \sum _{n=1}^4 \left( \frac{1}{q^{h_n}t_n} + \frac{1}{q^{l_n}t_n} \right) x^{-1} \right\} \Biggr] . \label{eq:qsecondal}
\end{gather}
The third and the second degenerate Ruijsenaars--van Diejen operators of one variable in \cite{TakqH} are realized as the case $\alpha =-1/2 $ in equations \eqref{eq:qthirdal} and \eqref{eq:qsecondal}, and conversely equations \eqref{eq:qthirdal} and~\eqref{eq:qsecondal} are obtained from the third and the second degenerate Ruijsenaars--van Diejen operators of one variable in \cite{TakqH} by appropriate gauge transformations.
The variant of the $q$-Heun equation of degree three is written as
\begin{equation*}
 A^{\langle 3 \rangle} ( x; h_1, h_2, h_3, l_1, l_2, l_3, \alpha , \beta ) g(x) =Eg(x), \qquad E \in {\mathbb C},
\end{equation*}
and the variant of the $q$-Heun equation of degree four is written as
\begin{equation*}
 A^{\langle 2 \rangle} ( x; h_1, h_2, h_3, h_4, l_1, l_2, l_3, l_4 , \alpha ) g(x) =Eg(x), \qquad E \in {\mathbb C}.
\end{equation*}
Note that the variant of the $q$-Heun equation of degree three is a $q$-deformation of the second order Fuchsian differential equation with four singularities $\{ t_1, t_2 , t_3 , \infty \} $ written as
\begin{equation*}
\frac{{\rm d}^2y}{{\rm d}z^2} + \left( \frac{\gamma}{z-t_1}+\frac{\delta }{z-t_2}+\frac{\epsilon}{z-t_3}\right) \frac{{\rm d}y}{{\rm d}z} + \frac{\alpha \beta z -B}{(z - t_1)(z - t_2)(z - t_3)} y = 0,
\end{equation*}
and the variant of the $q$-Heun equation of degree four is a $q$-deformation of the second order Fuchsian differential equation with four singularities $\{ t_1, t_2 , t_3 , t_4 \} $.

We obtain a kernel function and its identity for the operator $ A^{\langle 3 \rangle} ( x; h_1, h_2, h_3, l_1, l_2, l_3, \alpha , \beta ) $ as follows.
\begin{Theorem} \label{thm:A3KerId}
If the parameters satisfy
\begin{gather}
 \chi = \big( \tilde{h}_1 + \tilde{h}_2 + \tilde{h}_3 - \tilde{l}_1 - \tilde{l}_2 - \tilde{l}_3 - \tilde{\beta } \big)/2 , \qquad \nu = 2 \mu _0 + \alpha - \tilde{\alpha } + \chi ,\qquad \mu = \mu _0 + \chi +1, \nonumber \\
 \beta = -\tilde{\beta } - \chi , \qquad l_i = \tilde{h}_i +\mu _0 , \qquad h_i = \tilde{l}_i +\mu _0 +\chi , \qquad i=1,2,3, \label{eq:A3KerIdp}
\end{gather}
then we have
\begin{gather*}
 A^{\langle 3 \rangle} ( x; h_1, h_2, h_3, l_1, l_2, l_3, \alpha , \beta ) x^{-\alpha } s^{\chi +1 - \tilde{\alpha }} P_{\mu ,\mu _0 } (x,s) \\
\qquad = q^{\nu } A^{\langle 3 \rangle} \big( s; \tilde{h}_1, \tilde{h}_2, \tilde{h}_3, \tilde{l}_1, \tilde{l}_2, \tilde{l}_3, \tilde{\alpha } , \tilde{\beta }\big) x^{-\alpha } s^{\chi +1 - \tilde{\alpha }} P_{\mu ,\mu _0 } (x,s). \nonumber
\end{gather*}
\end{Theorem}
Theorem \ref{thm:A3KerId} is shown similarly to Theorem \ref{thm:A4KerId}.
Note that equation \eqref{eq:A3KerIdp} is equivalent to
\begin{gather*}
 \chi = ( h_1 + h_2 + h_3 - l_1 - l_2 - l_3 - \beta )/2 , \qquad \nu = 2 \mu _0 + \alpha - \tilde{\alpha } + \chi , \qquad \mu = \mu _0 + \chi +1, \\
 \tilde{\beta }= - \beta - \chi , \qquad \tilde{h}_i = l_i -\mu _0 , \qquad \tilde{l}_i = h_i -\mu _0 -\chi , \qquad i=1,2,3. \nonumber
\end{gather*}

We also obtain a kernel function and its identity for the operator $ A^{\langle 2 \rangle} ( x; h_1, h_2, h_3, h_4, l_1, l_2, l_3,\allowbreak l_4 ,\alpha ) $ as follows.
\begin{Theorem} \label{thm:A2KerId}
If the parameters satisfy
\begin{gather}
\chi = \big( \tilde{h}_1 + \tilde{h}_2 + \tilde{h}_3 + \tilde{h}_4 - \tilde{l}_1 - \tilde{l}_2 - \tilde{l}_3 - \tilde{l}_4 \big)/2 , \qquad \nu = 2 \mu _0 + \alpha - \tilde{\alpha } + \chi , \nonumber \\
 \mu = \mu _0 + \chi +1,\qquad l_i = \tilde{h}_i +\mu _0 , \qquad h_i = \tilde{l}_i +\mu _0 +\chi , \qquad i=1,2,3,4 , \label{eq:A2KerIdp}
\end{gather}
then we have
\begin{gather*}
 A^{\langle 2 \rangle} ( x; h_1, h_2, h_3, h_4, l_1, l_2, l_3, l_4 ,\alpha ) x^{-\alpha } s^{\chi +1 - \tilde{\alpha }} P_{\mu ,\mu _0 } (x,s) \\
 \qquad= q^{\nu } A^{\langle 2 \rangle} \alpha\big( s; \tilde{h}_1, \tilde{h}_2, \tilde{h}_3, \tilde{h}_4, \tilde{l}_1, \tilde{l}_2, \tilde{l}_3, \tilde{l}_4 , \tilde{\alpha } \alpha\big) x^{-\alpha } s^{\chi +1 - \tilde{\alpha }} P_{\mu ,\mu _0 } (x,s). \nonumber
\end{gather*}
\end{Theorem}
Theorem \ref{thm:A2KerId} is also shown similarly to Theorem \ref{thm:A4KerId}.
Note that equation \eqref{eq:A2KerIdp} is equivalent~to
\begin{gather*}
 \chi = ( h_1 + h_2 + h_3 + h_4 - l_1 - l_2 - l_3 - l_4 )/2 , \qquad \nu = 2 \mu _0 + \alpha - \tilde{\alpha } + \chi , \\
 \mu = \mu _0 + \chi +1,\qquad \tilde{h}_i = l_i -\mu _0 , \qquad \tilde{l}_i = h_i -\mu _0 -\chi , \qquad i=1,2,3,4 .
\end{gather*}

\section[Kernel function and q-integral transformation]{Kernel function and $\boldsymbol{q}$-integral transformation} \label{sec:qinttrans}

In this section, we obtain $q$-integral transformations by using the kernel function.

Let $\xi \in {\mathbb C} \setminus \{ 0 \}$.
The infinite sum defined by integral
\begin{equation*}
 \int^{\xi \infty }_{0}f(s) {\rm d}_{q}s = (1-q)\sum^{\infty}_{n=-\infty}q^{n} \xi f(q^n \xi )
\end{equation*}
is called the Jackson integral.
It is known that the usual integral over $(0 ,+ \infty )$ is recovered as~${q \to 1}$ (see, e.g., \cite{GR}).
\begin{Theorem} \label{thm:KFqint}
Assume that the function $\Phi (x,s)$ satisfies
\begin{gather}
 \bigl\{ a(x) T_{x}^{-1} + b(x) +c(x) T_{x} - \big( \tilde{a}(s) T_{s}^{-1} + \tilde{b}(s) + \tilde{c}(s) T_{s} \big) \bigr\} \Phi (x,s)= 0 . \label{eq:Phixs}
\end{gather}
If the function $h(s)$ satisfies
\begin{gather}
\big\{ q \tilde{a}(qs) T_{s} + \tilde{b}(s) + q^{-1} \tilde{c}(s/q) T_{s}^{-1} \big\} h(s) =0 , \label{eq:hseq}
\end{gather}
the Jackson integral
\begin{gather}
 g (x) := \int^{\xi \infty}_{0} h(s) \Phi (x, s) {\rm d}_{q}s \label{eq:gxJint}
\end{gather}
and $g\big(q^{\pm 1}x\big)$ converge, the limits
\begin{align}
& g_1 (x) := \lim _{L \to +\infty } q s \tilde{a}(q s) h(q s) \Phi (x,s) - s \tilde{c}(s) h(s) \Phi (x,qs) | _{s=q^{L} \xi } , \nonumber \\
& g_2 (x) := \lim _{K \to -\infty } q s \tilde{a}(q s) h(q s) \Phi (x,s) - s \tilde{c}(s) h(s) \Phi (x,qs) | _{s=q^{K-1} \xi } \label{eq:g1g2limLK}
\end{align}
converge, and the variable $\xi $ is independent of the variable $x$ or it is proportional to $x$ $($i.e., $\xi =Ax$ where $A $ is independent of $x)$, then the function $g(x)$ in equation \eqref{eq:gxJint} satisfies
\begin{equation*}
 \big\{ a(x) T_{x}^{-1} + b(x) +c(x) T_{x} \big\} g(x) = (1-q) ( g_2 (x) -g_1 (x)) .
\end{equation*}
\end{Theorem}
To obtain the theorem, we use the following proposition.
\begin{Proposition} \label{prop:gKL1}
Assume that the function $\Phi (x,s)$ satisfies equation \eqref{eq:Phixs} and the function~$h(s)$ satisfies equation \eqref{eq:hseq}.
Set
\begin{gather*}
 g ^{[K,L]}(x) = (1-q)\sum ^{L}_{n=K} s h(s) \Phi (x, s) | _{s=q^{n} \xi } .
\end{gather*}
\begin{itemize}\itemsep=0pt
\item[$(i)$] If the variable $\xi $ is independent of the variable $x$, then we have
\begin{gather}
 \big\{ a(x) T_{x}^{-1} + b(x) +c(x) T_{x} \big\} g ^{[K,L]}(x) \nonumber \\
\qquad = - (1-q) \big[ q s \tilde{a}(q s) h(q s) \Phi (x,s) - s \tilde{c}(s) h(s) \Phi (x,qs) \big] \big| _{s=q^{L} \xi } \nonumber \\
\phantom{ \qquad=}{} + (1-q) \big[ q s \tilde{a}(q s) h(q s) \Phi (x,s) - s \tilde{c}(s) h(s) \Phi (x,qs) \big] \big| _{s=q^{K-1} \xi } . \label{eq:gKL1}
\end{gather}
\item[$(ii)$] If $\xi =A x $ and $A$ is independent of $x$, then we have
\begin{gather}
 \big\{ a(x) T_{x}^{-1} + b(x) +c(x) T_{x} \big\} g ^{[K,L]}(x) \nonumber \\
 \qquad= - (1-q) [ s a(x) h(s) \Phi (x/q, s ) - q s c(x) h(q s ) \Phi (qx,q s ) ] _{s=q^{L} A x } \nonumber \\
\phantom{ \qquad=}{} + (1-q) [ s a(x) h(s) \Phi (x/q, s ) - q s c(x) h (q s) \Phi (qx, q s ) ] _{s=q^{K-1} A x } \nonumber \\
 \phantom{ \qquad=}{} - (1-q) [ q s \tilde{a}(q s) h(q s) \Phi (x,s) - s \tilde{c}(s) h(s) \Phi (x,qs) ] | _{s=q^{L} A x } \nonumber \\
 \phantom{ \qquad=}{} + (1-q) [ q s \tilde{a}(q s) h(q s) \Phi (x,s) - s \tilde{c}(s) h(s) \Phi (x,qs) ] | _{s=q^{K-1} A x } . \label{eq:gKL2}
\end{gather}
\end{itemize}
\end{Proposition}
\begin{proof}
It follows from equation \eqref{eq:Phixs} that
\begin{align}
& [ a(x) \Phi (x/q ,s) + b(x) \Phi (x,s) +c(x) \Phi (qx,s) ]_{s = q^n \xi} \nonumber \\
& \qquad{}= [ \tilde{a}(s) \Phi (x,s/q) + \tilde{b}(s) \Phi (x,s) + \tilde{c}(s) \Phi (x,q s) ]_{s = q^n \xi } . \label{eq:axPhiasPhi}
\end{align}

(i) If $\xi $ is independent of $x$, then we have
\begin{gather*}
 \big\{ a(x) T_{x}^{-1} + b(x) +c(x) T_{x} \big\} \sum _{n=K}^L s \Phi (x,s) h(s) | _{s=q^n \xi } \\
\qquad = \sum _{n=K}^L \big[ s \tilde{a}(s) \Phi (x,s/q) h(s) + s \tilde{b}(s) \Phi (x,s) h(s) + s \tilde{c}(s) \Phi (x,qs) h(s) \big] | _{s=q^n \xi } \\
\qquad = \sum _{n=K-1}^{L-1} q s \tilde{a}(q s) \Phi (x,s) h(q s) | _{s=q^n \xi } + \sum _{n=K}^L s \tilde{b}(s) \Phi (x,s) h(s) | _{s=q^n \xi } \\
 \phantom{ \qquad=}{} + \sum _{n=K+1}^{L+1} q^{-1} s \tilde{c}(s/q) \Phi (x,s) h(s/q) | _{s=q^n \xi } \\
\qquad = \sum _{n=K}^L \big[ q \tilde{a}(q s) h(q s) + \tilde{b}(s) h(s) + q^{-1} \tilde{c}(s/q) h(s/q) \big] s \Phi (x,s) | _{s=q^n \xi } \\
 \phantom{ \qquad=}{} - q s \tilde{a}(q s) \Phi (x,s) h(q s) | _{s=q^L \xi } + q^{-1} s \tilde{c}(s/q) \Phi (x,s) h(s/q) | _{s=q^{L+1} \xi } \\
 \phantom{ \qquad=}{} + q s \tilde{a}(q s) \Phi (x,s) h(q s) | _{s=q^{K-1} \xi } - q^{-1} s \tilde{c}(s/q) \Phi (x,s) h(s/q) | _{s=q^{K} \xi } .
\end{gather*}
Therefore, we obtain (i) by equation \eqref{eq:hseq}.

(ii) If $\xi =A x$ and $A$ is independent of $x$, then it follows from equation \eqref{eq:axPhiasPhi} that
\begin{gather*}
 \big\{ a(x) T_{x}^{-1} + b(x) +c(x) T_{x} \big\} \sum _{n=K}^L s \Phi (x,s) h(s) | _{s=q^n A x } \\
\qquad = \sum _{n=K}^L \big\{ a(x) q^{n-1} A x \Phi \big(x/q ,q^{n-1} A x\big) h\big(q^{n-1} A x\big) + b(x) q^n A x \Phi \big(x,q^n A x\big) h\big(q^n A x\big) \\
 \phantom{ \qquad=}{} +c(x) q^{n+1} A x \Phi \big(qx,q^{n+1} A x\big) h\big(q^{n+1} A x\big) \big\} \\
\qquad = \sum _{n=K}^L s h(s) ( a(x) \Phi (x/q,s) + b(x) \Phi (x,s) + c(x) \Phi (qx,s)) | _{s=q^n A x } \\
 \phantom{ \qquad=}{} + a(x) q^{K-1} A x \Phi \big(x/q,q^{K-1} A x\big) h\big(q^{K-1} A x\big) - a(x) q^{L} A x \Phi \big(x/q,q^{L} A x\big) h\big(q^{L} A x\big) \\
 \phantom{ \qquad=}{} - c(x) q^{K} A x \Phi \big(qx,q^{K} A x\big) h \big(q^{K} A x\big) + c(x) q^{L+1} A x \Phi \big(qx,q^{L+1} A x\big) h\big(q^{L+1} A x\big) \\
\qquad = \sum _{n=K}^L \big[ s \tilde{a}(s) \Phi (x,s/q) h(s) + s \tilde{b}(s) \Phi (x,s) h(s) + s \tilde{c}(s) \Phi (x,qs) h(s) \big] \big| _{s=q^n A x } \\
 \phantom{ \qquad=}{} - [ s a(x) \Phi (x/q, s ) h(s) - q s c(x) \Phi (qx,q s ) h(q s ) ] _{s=q^{L} A x } \\
 \phantom{ \qquad=}{} + [ s a(x) \Phi (x/q, s ) h(s) - q s c(x) \Phi (qx, q s ) h (q s) ] _{s=q^{K-1} A x } ,
\end{gather*}
where we applied equation \eqref{eq:axPhiasPhi} in the case $\xi =A x $.
Hence we obtain (ii) by repeating the discussion in (i).
\end{proof}

We continue the proof of Theorem \ref{thm:KFqint}.
If the variable $\xi $ is independent of the variable $x$, then we obtain Theorem \ref{thm:KFqint} from equation \eqref{eq:gKL1} as $L \to +\infty $ and $K \to -\infty $.

We consider the case $\xi =A x $.
The convergence of the Jackson integrals $g\big(q^{\pm 1} x\big)$ is equivalent to the convergence of the summations
\begin{gather*}
\sum _{n=0}^{L } s \Phi \big( q^{\pm 1} x,s\big) h(s) | _{s=q^n A x } , \qquad \sum _{n=K}^{-1} s \Phi \big( q^{\pm 1} x,s\big) h(s) | _{s=q^n A x }
\end{gather*}
as $L \to +\infty $ and $K \to -\infty $, and it follows that
\begin{gather*}
 \lim _{L \to +\infty } s \Phi \big( q^{\pm 1} x,s\big) h(s) | _{s=q^n A x } = 0 = \lim _{K \to -\infty } s \Phi \big( q^{\pm 1} x,s\big) h(s) | _{s=q^n A x } .
\end{gather*}
Therefore, we have
\begin{gather*}
 \lim _{n \to \pm \infty } [ s a(x) \Phi (x/q, s ) h(s) - q s c(x) \Phi (qx,q s ) h(q s ) ] _{s=q^{n} A x } =0 ,
\end{gather*}
and we obtain Theorem \ref{thm:KFqint} for the case $\xi =A x $ from equation \eqref{eq:gKL2} as $L \to +\infty $ and $K \to -\infty $.

\section[q-integral transformations for q-Heun equations and its variants]{$\boldsymbol{q}$-integral transformations for $\boldsymbol{q}$-Heun equations\\ and its variants} \label{sec:qintHeunvar}

In this section, we obtain $q$-integral transformations for the $q$-Heun equation and its variants by applying results in Sections~\ref{sec:Kerfunc} and~\ref{sec:qinttrans}.

%\subsection{}

Recall that the operator $A^{\langle 4 \rangle} ( x; h_1, h_2, l_1, l_2, \alpha _1, \alpha _2, \beta ) $ was introduced in equation \eqref{eq:A4op}, and the $q$-Heun equation was written as the equation for the eigenfunction of it with the eigenvalue~$E$.
In Theorem \ref{thm:A4KerId}, we obtained the identity
\begin{align}
& A^{\langle 4 \rangle} ( x; h_1, h_2, l_1, l_2, \alpha _1, \alpha _2, \beta ) x^{- \alpha _1} s^{1+\chi - \tilde{\alpha }_1} P_{\mu ,\mu _0 } (x,s) \nonumber \\
&\qquad = q^{\nu } A^{\langle 4 \rangle} \big( s; \tilde{h}_1, \tilde{h}_2, \tilde{l}_1, \tilde{l}_2, \tilde{\alpha }_1, \tilde{\alpha }_2, \tilde{\beta } \big) x^{- \alpha _1} s^{1+\chi - \tilde{\alpha }_1} P_{\mu ,\mu _0 } (x,s) , \label{eq:A4KerId1}
\end{align}
where the parameters satisfy equation \eqref{eq:A4KerIdp}.

We now apply Theorem \ref{thm:KFqint} for the $q$-Heun equation.
Let $E $ and $ \tilde{E} $ be constants such that~${E = q^{\nu } \tilde{E}}$ and write
\begin{align*}
& a(x) T_{x}^{-1} + b(x) +c(x) T_{x}= A^{\langle 4 \rangle} ( x; h_1, h_2, l_1, l_2, \alpha _1, \alpha _2, \beta ) -E , \\
& \tilde{a}(s) T_{s}^{-1} + \tilde{b}(s) + \tilde{c}(s) T_{s} = q^{\nu } \big\{ A^{\langle 4 \rangle} \big( s; \tilde{h}_1, \tilde{h}_2, \tilde{l}_1, \tilde{l}_2, \tilde{\alpha }_1 , \tilde{\alpha }_2, \tilde{\beta } \big) - \tilde{E} \big\} . \nonumber
\end{align*}
Equation \eqref{eq:A4KerId1} gives a realization of equation \eqref{eq:Phixs} by setting $ \Phi (x,s) = x^{- \alpha _1} s^{1+\chi - \tilde{\alpha }_1} P_{\mu ,\mu _0 } (x,s)$.
In this situation, the coefficients of the equation
\begin{gather}
\big\{ q \tilde{a}(qs) T_{s} + \tilde{b}(s) + q^{-1} \tilde{c}(s/q) T_{s}^{-1}\big\} h(s) = 0 \label{eq:A4hs}
\end{gather}
is described as
\begin{align}
& q\tilde{a}(qs)=q^{\nu + 2} s^{-1} \big(s-q^{\tilde{h}_1 - 1/2} t_1\big) \big(s-q^{\tilde{h}_2 -1/2} t_2\big) , \nonumber \\
& q^{-1} \tilde{c}(s/q) = q^{\nu + \tilde{\alpha } _1 +\tilde{\alpha }_2-2} s^{-1} \big(s - q^{\tilde{l}_1 +1/2} t_1\big) \big(s - q^{\tilde{l}_2 + 1/2} t_2\big) , \nonumber \\
& \tilde{b}(s) = -q^{\nu } \big\{ \big(q^{\tilde{\alpha } _1} +q^{\tilde{\alpha } _2} \big) s + q^{(\tilde{h}_1 +\tilde{h}_2 + \tilde{l}_1 +\tilde{l}_2 +\tilde{\alpha } _1 +\tilde{\alpha } _2 )/2} \big( q^{\tilde{\beta }/2} + q^{-\tilde{\beta }/2} \big) t_1 t_2 s^{-1} \big\} - q^{\nu } \tilde{E} . \label{eq:A4hscond}
\end{align}
Equation \eqref{eq:A4hs} is equivalent to the equation $A^{\langle 4 \rangle} \big( x; h'_1, h'_2, l'_1, l'_2, \alpha '_1, \alpha '_2, \beta ' \big) h(s) = E' h(s) $, where
\begin{align}
 \alpha ' _i = 2 - \tilde{\alpha } _i, \qquad l'_i =\tilde{h}_i , \qquad h'_i = \tilde{l}_i , \qquad i=1,2 , \qquad \beta ' = \tilde{\beta }, \qquad E' = q^{2 - \tilde{\alpha } _1 - \tilde{\alpha } _2} \tilde{E } . \label{eq:A4til'}
\end{align}
Note that the exponents of equation \eqref{eq:A4hs} about $s=\infty $ are $\{ \alpha ' _1 , \alpha ' _2 \}$ and the exponents about~${s=0}$ are $\{ \lambda ' _+, \lambda '_- \}$, where $\lambda '_{\pm } = \big( h'_1 +h'_2 -l'_1-l'_2 -\alpha '_1-\alpha '_2 \pm \beta '+2\big)/2$.
See \cite{TakqH} for the definition and details of the exponents for the $q$-difference equation.
In particular, if $\alpha '_1 - \alpha ' _2 \not \in {\mathbb Z} $ (resp.\ $\beta ' \not \in {\mathbb Z} $), then there exist the solutions $h_1 (s) $ and $h_2 (s)$ (resp.\ $h_3 (s) $ and $h_4 (s)$) of equation \eqref{eq:A4hs} about $s=\infty $ (resp.\ $s=0$) such that $h_1 (s) / s^{-\alpha '_1} \to 1$ and $h_2 (s) / s^{-\alpha '_2} \to 1$ as $s \to \infty $ \big(resp.\ $h_3 (s) / s^{\lambda ' _+} \to 1$ and $h_4 (s) / s^{\lambda '_- } \to 1$ as $s \to 0 $\big).
The values $\lambda ' _+ \big( = \big(h'_1 + h'_2 - l'_1 - l'_2 - \alpha '_1 - \alpha '_2 + \beta ' +2\big)/2 \big)$ and $-\alpha '_1$ will appear in Theorem \ref{thm:A4}.
As a~consequence of Theorem \ref{thm:KFqint}, we obtain the following theorem.
\begin{Theorem} \label{thm:A4}
Assume that the function $h(s)$ satisfies
\begin{gather}
 A^{\langle 4 \rangle} \big( s; h'_1, h'_2, l'_1, l'_2, \alpha '_1, \alpha '_2, \beta ' \big) h(s) = E' h(s) \label{eq:A4hsE'}
\end{gather}
and
\begin{gather}
 \lim _{L \to +\infty } \frac{h(s)}{s^{(h'_1 + h'_2 - l'_1 - l'_2 - \alpha '_1 - \alpha '_2 + \beta ' +2)/2}} \big| _{s=q^{L} \xi } = C_1, \qquad \lim _{K \to -\infty } \frac{h(s)}{s^{-\alpha ' _1}} \big| _{s=q^{K} \xi } = C_2 \label{eq:A4hslim}
\end{gather}
for some constants $E'$, $C_1$, $C_2$.
Then the Jackson integral
\begin{gather}
 g (x) = x^{- \alpha _1} \int^{\xi \infty }_{0} s^{-(h'_1 + h'_2 - l'_1 - l'_2 - \alpha '_1 - \alpha '_2 + \beta ' +2)/2} h(s) P^{(1)}_{\mu ,\mu _0 } (x,s) {\rm d}_{q}s \label{eq:A4Jint}
\end{gather}
converges and it satisfies
\begin{gather}
 A^{\langle 4 \rangle} ( x; h_1, h_2, l_1, l_2, \alpha _1, \alpha _2, \beta ) g (x) = E g(x) + (1-q) (g_2 (x) - g_1 (x)), \label{eq:thmA4gEgg2g1}
\end{gather}
where
\begin{gather}
 g_1 (x)= C_1 x^{- \alpha _1} q^{\mu _0 + \alpha _1 + h'_1 + h'_2 + \chi } \big( q^{ \beta ' } - 1 \big) t_1 t_2 , \nonumber \\
 g_2(x)= C_2 x^{- \alpha _1 } \frac{\vartheta _q (q^{- \mu _0 - \chi } x/\xi )}{\vartheta _q \big(q^{- \mu _0 +1 } x/\xi \big)} \xi ^{\chi + 1 } q^{\mu _0 + \alpha _1 } \big( q^{\alpha ' _2 - \alpha '_1} - 1 \big) , \nonumber \\
 \chi = \big( l'_1 + l'_2 - h'_1 - h'_2 - \alpha '_1 + \alpha '_2 - \beta ' \big)/2 , \qquad \mu = \mu _0 + 1+ \chi , \nonumber \\
 E = q^{\mu _0 + \alpha _1 - \alpha ' _1 } E', \qquad \beta = - \beta ' - \chi , \qquad \alpha _2 = \alpha _1 - \alpha ' _1 + \alpha ' _2 - \chi , \nonumber \\
 l_i = l'_i +\mu _0 , \qquad h_i = h'_i +\mu _0 +\chi , \qquad i=1,2. \label{eq:thmA4param}
\end{gather}
In particular, if $C_1=C_2=0$, then we have
\begin{gather*}
 A^{\langle 4 \rangle} ( x; h_1, h_2, l_1, l_2, \alpha _1, \alpha _2, \beta ) g (x) = E g(x) .
\end{gather*}
\end{Theorem}
\begin{proof}
On the given parameters $h'_1$, $h'_2$, $l'_1$, $l'_2$, $\alpha '_1$, $\alpha '_2$, $\beta '$, $E' $, we define the parameters $\tilde{h}_1$, $\tilde{h}_2$, $\tilde{l}_1$, $\tilde{l}_2$, $\tilde{\alpha }_1$, $\tilde{\alpha }_2$, $\tilde{\beta }$, $\tilde{E} $ by equation \eqref{eq:A4til'}.
Then it follows from equation \eqref{eq:A4hsE'} that the function $h(s)$ satisfies equation \eqref{eq:A4hs} with equation \eqref{eq:A4hscond}.
By combining equation \eqref{eq:thmA4param} with equation \eqref{eq:A4til'}, we obtain equation \eqref{eq:A4KerIdp}.

It follows from equation \eqref{eq:A4hslim} that \smash{$h(s) s^{1+ \chi - \tilde{\alpha }_1} | _{s=q^{L} \xi } \to C_1 $} as $L \to +\infty $.
On the first limit in equation \eqref{eq:g1g2limLK}, we have
\begin{gather*}
 x^{- \alpha _1} \big[ q s \tilde{a}(q s) h(qs) s^{1+\chi - \tilde{\alpha }_1} P^{(1)}_{\mu ,\mu _0 } (x,s) - s \tilde{c}(s) h(s) (qs)^{1+\chi - \tilde{\alpha }_1} P^{(1)}_{\mu ,\mu _0 } (x,qs) \big] \big| _{s=q^{L} \xi } \\
\qquad = q^{\nu } x^{- \alpha _1} \big[ \big(qs-q^{\tilde{h}_1 + 1/2} t_1\big) \big(qs-q^{\tilde{h}_2 +1/2} t_2\big) q^{-1- \chi + \tilde{\alpha }_1} P^{(1)}_{\mu ,\mu _0 } (x,s) h(qs) (qs)^{1+ \chi - \tilde{\alpha }_1} \\
\phantom{\qquad=}{} - q^{\tilde{\alpha } _1 +\tilde{\alpha }_2} \big(s - q^{\tilde{l}_1 -1/2} t_1\big) \big(s - q^{\tilde{l}_2 -1/2} t_2\big) q^{1+\chi - \tilde{\alpha }_1} P^{(1)}_{\mu ,\mu _0 } (x,qs) h(s) s^{1+ \chi - \tilde{\alpha }_1 } \big] \big| _{s=q^{L} \xi } \\
\qquad {} \to C_1 x^{- \alpha _1} q^{\nu } \big( q^{\tilde{h}_1 + \tilde{h}_2 - \chi + \tilde{\alpha }_1} - q^{\tilde{\alpha } _1 +\tilde{\alpha }_2 + \tilde{l}_1 + \tilde{l}_2 +\chi - \tilde{\alpha }_1} \big) t_1 t_2 \\
\qquad = C_1 x^{- \alpha _1} q^{ \mu _0 + \alpha _1 + (\tilde{h}_1 + \tilde{h}_2 + \tilde{l}_1 + \tilde{l}_2 + \tilde{\alpha }_1 - \tilde{\alpha }_2 )/2} \big( q^{ \tilde{\beta }/2} - q^{- \tilde{\beta } /2} \big) t_1 t_2
\end{gather*}
as $L \to +\infty $. Here we used \smash{$P^{(1)}_{\mu ,\mu _0 } (x,s) \to 1$} as $s \to 0 $ (see equation \eqref{eq:P12lim1}).
Therefore, the first limit in equation \eqref{eq:g1g2limLK} converges and we have
\begin{equation*}
g_1 (x)= C_1 x^{- \alpha _1} q^{\mu _0 + \alpha _1 + (h'_1 + h'_2 + l'_1 + l'_2 -\alpha '_1 + \alpha '_2 )/2} \big( q^{ \beta '/2} - q^{- \beta ' /2} \big) t_1 t_2 .
\end{equation*}
It also follows from equation \eqref{eq:A4hslim} that $h(s) s^{2- \tilde{\alpha }_1} | _{s=q^{K-1} \xi } \to C_2 $ as $K \to -\infty $.
On the second limit in equation \eqref{eq:g1g2limLK}, we have
\begin{gather*}
 x^{- \alpha _1} \big[ q s \tilde{a}(q s) h(qs) s^{1+\chi - \tilde{\alpha }_1} P^{(1)}_{\mu ,\mu _0 } (x,s) - s \tilde{c}(s) h(s) (qs)^{1+\chi - \tilde{\alpha }_1} P^{(1)}_{\mu ,\mu _0 } (x,qs) \big] \big| _{s=q^{K-1} \xi } \\
 \qquad= x^{- \alpha _1} (x/\xi )^{\mu _0 - \mu } \frac{\vartheta _q \big(q^{- \mu +1 } x/\xi \big)}{\vartheta _q \big(q^{- \mu _0 +1 } x/\xi \big)} \big[ q^{\tilde{\alpha }_1 -1} s^{-1} \tilde{a}(q s) s^{1+\chi } P^{(2)}_{\mu ,\mu _0 } (x,s) h(qs) (qs)^{2 - \tilde{\alpha }_1} \\
 \phantom{\qquad=}{} - s^{-1} \tilde{c}(s) q^{- \tilde{\alpha }_1} (qs)^{1+\chi } P^{(2)}_{\mu ,\mu _0 } (x,qs) h(s) s^{ 2 - \tilde{\alpha }_1} \big] \big| _{s=q^{K-1} \xi } \\
\qquad {}\to x^{- \alpha _1} (x/\xi )^{-1-\chi } \frac{\vartheta _q \big(q^{- \mu +1 } x/\xi \big)}{\vartheta _q \big(q^{- \mu _0 +1 } x/\xi \big)}q^{\nu } \big( q^{\tilde{\alpha }_1 } x^{1+\chi } C_2 - q^{\tilde{\alpha }_2} x^{1+\chi } C_2 \big) \\
\qquad = C_2 x^{- \alpha _1} \xi ^{1+\chi } \frac{\vartheta _q \big(q^{- \mu _0 - \chi } x/\xi \big)}{\vartheta _q \big(q^{- \mu _0 +1 } x/\xi \big)} q^{\mu _0 + \alpha _1 - \tilde{\alpha } _2 } \big( q^{\tilde{\alpha }_1 } - q^{\tilde{\alpha }_2} \big)
\end{gather*}
as $K \to -\infty $. Here we used equation \eqref{eq:P12lim1}.
Therefore, the second limit in equation \eqref{eq:g1g2limLK} converges and we have
\begin{equation*}
g_2(x)= C_2 x^{- \alpha _1 } \xi ^{\chi + 1 } \frac{\vartheta _q (q^{- \mu _0 - \chi } x/\xi )}{\vartheta _q \big(q^{- \mu _0 +1 } x/\xi \big)} q^{\mu _0 + \alpha _1 } \big( q^{\alpha ' _2 - \alpha '_1} - 1 \big) .
\end{equation*}
We show convergence of the Jackson integral in equation \eqref{eq:A4Jint}, which is equivalent to convergence of the summation of the sequence $a_n$ over $n \in {\mathbb Z} $, where
\begin{gather}
a_n = s h(s) s^{1+\chi - \tilde{\alpha }_1} P^{(1)}_{\mu ,\mu _0 } (x,s) | _{s=q^{n} \xi } .
\label{eq:A4an}
\end{gather}
Since $h(s) s^{1+ \chi - \tilde{\alpha }_1} | _{s=q^{L} \xi } \to C_1 $ as $L \to +\infty $ and \smash{$P^{(1)}_{\mu ,\mu _0 } (x,s) \to 1$} as $s \to 0 $, there exists an integer~$N_1$ and a positive number $C'_1$ such that $| a_n | \leq C'_1 |q|^n $ for any integer $n$ such that $n \geq N_1$.
Hence, we obtain convergence of the summation of $a_n$ over $n \in {\mathbb Z} _{\geq 0}$, because $0 <|q| <1$.
Equation \eqref{eq:A4an} is also expressed as
\begin{gather*}
a_n = (x/\xi )^{\mu _0 - \mu } \frac{\vartheta _q \big(q^{- \mu +1 } x/\xi \big)}{\vartheta _q \big(q^{- \mu _0 +1 } x/\xi \big)} s^{-1} h(s) s^{2 - \tilde{\alpha }_1} s^{1+\chi } P^{(2)}_{\mu ,\mu _0 } (x,s) | _{s=q^{n} \xi } .
\end{gather*}
Since $h(s) s^{2- \tilde{\alpha }_1} | _{s=q^{K-1} \xi } \to C_2 $ as $K \to -\infty $ and \smash{$s^{\mu - \mu _0 } P^{(2)}_{\mu ,\mu _0 } (x,s) \to x^{\mu - \mu _0 }$} as $s \to \infty $, there exists an integer $N_2$ and a positive number $C'_2$ such that $| a_n | \leq C'_2 |q|^{-n} $ for any integer $n$ such that $n \leq N_2$.
Hence, we obtain convergence of the summation of $a_n$ over $n \in {\mathbb Z} _{\leq -1}$.
Therefore, convergence of the Jackson integral in equation \eqref{eq:A4Jint} is shown.

Recall that equation \eqref{eq:A4KerId1} gives a realization of equation \eqref{eq:Phixs} and we have the relation $E= q^{\mu _0 + \alpha _1 - \alpha ' _1 } E' = q^{\nu } \tilde{E} $ by equations \eqref{eq:A4til'} and \eqref{eq:thmA4param}.
Hence, we have confirmed the assumption of Theorem \ref{thm:KFqint}, and we obtain that the function $g(x)$ satisfies equation \eqref{eq:thmA4gEgg2g1}.
\end{proof}

In Theorem~\ref{thm:A4}, we can replace the Jackson integral in equation \eqref{eq:A4Jint} with
\begin{gather}
 g (x) = x^{- \alpha _1} \int^{\xi \infty }_{0} s^{-(h'_1 + h'_2 - l'_1 - l'_2 - \alpha '_1 - \alpha '_2 + \beta ' +2)/2} h(s) P^{(2)}_{\mu ,\mu _0 } (x,s) {\rm d}_{q}s . \label{eq:A4Jintv2}
\end{gather}
Namely, under the assumption of Theorem \ref{thm:A4}, the function $g(x)$ in equation \eqref{eq:A4Jintv2}
converges and it satisfies
\begin{equation*}
 A^{\langle 4 \rangle} ( x; h_1, h_2, l_1, l_2, \alpha _1, \alpha _2, \beta ) g (x) = E g(x) + (1-q) (g_2 (x) - g_1 (x)),
\end{equation*}
where
\begin{gather*}
 g_1 (x)= C_1 x^{- \alpha _1+ \chi +1} \xi ^{ -\chi -1 } \frac{\vartheta _q \big(q^{- \mu _0 +1 } x/\xi \big)}{\vartheta _q (q^{- \mu _0 -\chi } x/\xi )} q^{\mu _0 + \alpha _1 + h'_1 + h'_2 + \chi } \big( q^{ \beta ' } - 1 \big) t_1 t_2 , \\
 g_2(x)= C_2 x^{- \alpha _1 + \chi +1} q^{\mu _0 + \alpha _1 } \big( q^{\alpha ' _2 - \alpha '_1} - 1 \big) , \\
 \chi = \big( l'_1 + l'_2 - h'_1 - h'_2 - \alpha '_1 + \alpha '_2 - \beta ' \big)/2 , \qquad \mu = \mu _0 + 1+ \chi , \\
 E = q^{\mu _0 + \alpha _1 - \alpha ' _1 } E', \qquad \beta = - \beta ' - \chi , \qquad \alpha _2 = \alpha _1 - \alpha ' _1 + \alpha ' _2 - \chi , \\
 l_i = l'_i +\mu _0 , \qquad h_i = h'_i +\mu _0 +\chi , \qquad i=1,2.
\end{gather*}
This is proved by rewriting the proof of Theorem \ref{thm:A4} while replacing \smash{$P^{(2)}_{\mu ,\mu _0 } (x,s) $} with \smash{$P^{(1)}_{\mu ,\mu _0 } (x,s) $} by using equation \eqref{eq:P2P1theta}.

A $q$-integral transformation of the $q$-Heun equation was established in \cite[Theorem 5.4]{STT2}, and the convergence was not considered well there.
By specializing the parameters in Theorem \ref{thm:A4}, we reproduce \cite[Theorem 5.4]{STT2} with attention to the convergence as the following corollary.
\begin{Corollary} \label{cor:A4}
Assume that $ h'_1 + h'_2 - l'_1 - l'_2 - \alpha '_1 - \alpha '_2 + \beta ' +2=0$ and the function $h(s)$ satisfies
\begin{equation*}
 A^{\langle 4 \rangle} \big( s; h'_1, h'_2, l'_1, l'_2, \alpha '_1, \alpha '_2, \beta ' \big) h(s) = E' h(s)
\end{equation*}
and
\begin{equation*}
 \lim _{L \to +\infty } h(s) \big| _{s=q^{L} \xi } = 0, \qquad \lim _{K \to -\infty } \frac{h(s)}{s^{-\alpha ' _1}} \big| _{s=q^{K} \xi } = 0
\end{equation*}
for some constant $E' $.
Then the Jackson integral
\begin{equation*}
 g (x) = \int^{\xi \infty }_{0} h(s) P^{(1)}_{2-\alpha '_1 , 0} (x,s) {\rm d}_{q}s = \int^{\xi \infty }_{0} h(s)\frac{\big(q^{2-\alpha '_1 }s/x;q\big)_{\infty }}{(s/x;q)_{\infty }} {\rm d}_{q}s
\end{equation*}
converges and it satisfies
\begin{equation*}
 A^{\langle 4 \rangle} \big( x; h'_1 +1 -\alpha '_1 , h'_2 +1 -\alpha '_1 , l'_1, l'_2, 0, \alpha ' _2 - 1, \beta ' + 1 - \alpha '_1 \big) g (x) = q^{ - \alpha ' _1 } E'g(x) .
\end{equation*}
\end{Corollary}
\begin{proof}
Set $\mu _0 =0$ and $\alpha _1 =0$ in Theorem \ref{thm:A4}.
It follows from $ h'_1 + h'_2 - l'_1 - l'_2 - \alpha '_1 - \alpha '_2 + \beta ' +2=0$ that $\chi = 1 -\alpha '_1 $ and $\mu = 2-\alpha ' _1 $, and we obtain the corollary from Theorem \ref{thm:A4}.
\end{proof}

%\subsection{}

We investigate a $q$-integral transformation for the variant of the $q$-Heun equation of degree three.
The operator $A^{\langle 3 \rangle} ( x; h_1, h_2, h_3, l_1, l_2, l_3, \alpha , \beta ) $ in equation \eqref{eq:qthirdal} was used to write down the variant of the $q$-Heun equation of degree three.
In Theorem \ref{thm:A3KerId}, we obtained the identity
\begin{gather}
A^{\langle 3 \rangle} ( x; h_1, h_2, h_3, l_1, l_2, l_3, \alpha , \beta ) x^{-\alpha } s^{\chi +1 - \tilde{\alpha }} P_{\mu ,\mu _0 } (x,s) \nonumber \\
\qquad = q^{\nu } A^{\langle 3 \rangle} \big( s; \tilde{h}_1, \tilde{h}_2, \tilde{h}_3, \tilde{l}_1, \tilde{l}_2, \tilde{l}_3, \tilde{\alpha } , \tilde{\beta }\big) x^{-\alpha } s^{\chi +1 - \tilde{\alpha }} P_{\mu ,\mu _0 } (x,s), \label{eq:A3KerId1}
\end{gather}
where the parameters satisfy equation \eqref{eq:A3KerIdp}.
We now apply Theorem \ref{thm:KFqint} for the variant of the $q$-Heun equation of degree three.
Let $E $ and $ \tilde{E} $ be constants such that $E = q^{\nu } \tilde{E} $ and write
\begin{align*}
& a(x) T_{x}^{-1} + b(x) +c(x) T_{x}= A^{\langle 3 \rangle} ( x; h_1, h_2, h_3, l_1, l_2, l_3, \alpha , \beta ) -E , \\
& \tilde{a}(s) T_{s}^{-1} + \tilde{b}(s) + \tilde{c}(s) T_{s} = q^{\nu } \big\{ A^{\langle 3 \rangle} \big( s; \tilde{h}_1, \tilde{h}_2, \tilde{h}_3, \tilde{l}_1, \tilde{l}_2, \tilde{l}_3, \tilde{\alpha } , \tilde{\beta } \big) - \tilde{E} \big\} . \nonumber
\end{align*}
Equation \eqref{eq:A3KerId1} gives a realization of equation \eqref{eq:Phixs} by setting $ \Phi (x,s) = x^{-\alpha } s^{\chi +1 - \tilde{\alpha }} P_{\mu ,\mu _0 } (x,s)$.
In this situation, the coefficients of the equation
\begin{gather}
\big\{ q \tilde{a}(qs) T_{s} + \tilde{b}(s) + q^{-1} \tilde{c}(s/q) T_{s}^{-1} \big\} h(s) = 0 \label{eq:A3hs}
\end{gather}
are described as
\begin{gather*}
 q\tilde{a}(qs) = q^{\nu +3 } s^{-1} \big(s-q^{\tilde{h}_1 - 1/2} t_1\big) \big(s-q^{\tilde{h}_2 -1/2} t_2\big) \big(s-q^{\tilde{h}_3 -1/2} t_3\big) , \\
 q^{-1} \tilde{c}(s/q) = q^{\nu +2\tilde{\alpha } -2} s^{-1} \big(s - q^{\tilde{l}_1 +1/2} t_1\big) \big(s - q^{\tilde{l}_2 + 1/2} t_2\big) \big(s - q^{\tilde{l}_3 + 1/2} t_3\big) , \\
 \tilde{b}(s) = q^{\nu +\tilde{\alpha } +1/2} \Biggl[ -\big(q^{1/2} +q^{-1/2} \big) s^2 +\sum _{n=1}^3 \big( q^{\tilde{h}_n} + q^{\tilde{l}_n} \big)t_n s  \\
  \phantom{ \tilde{b}(s) =}{}+ q^{(\tilde{l}_1 +\tilde{l}_2 +\tilde{l}_3 +\tilde{h}_1 +\tilde{h}_2 +\tilde{h}_3)/2} \big( q^{\tilde{\beta }/2} + q^{-\tilde{\beta }/2} \big) t_1 t_2 t_3 s^{-1} \Biggr] - q^{\nu } \tilde{E} .
\end{gather*}
Equation \eqref{eq:A3hs} is equivalent to the equation $A^{\langle 3 \rangle} \big( x; h'_1, h'_2, h'_3, l'_1, l'_2, l'_3, \alpha ', \beta ' \big) h(s) = E' h(s) $, where
\begin{gather*}
 l'_i =\tilde{h}_i , \qquad h'_i = \tilde{l}_i, \qquad i=1,2,3 , \qquad \alpha ' = 2 - \tilde{\alpha } , \qquad \beta ' = \tilde{\beta }, \qquad E' = q^{2 - 2 \tilde{\alpha } } \tilde{E } .
\end{gather*}
Note that the exponents of equation \eqref{eq:A3hs} about $s=\infty $ are $\{ \alpha ' , \alpha ' + 1 \}$, the exponents about~${s=0}$ are $\{ \lambda ' _+, \lambda '_- \}$, where $\lambda '_{\pm } = \big( h'_1 +h'_2 +h'_3 -l'_1-l'_2 -l'_3 - 2 \alpha ' \pm \beta '+2\big)/2$, and the singular point $s=\infty $ is non-logarithmic.
The values $\lambda ' _+ \bigl( = \big( h'_1 +h'_2 +h'_3 -l'_1-l'_2 -l'_3 - 2 \alpha ' + \beta '+2\big)/2 \bigr)$ and $\alpha ' +1$ will appear in Theorem \ref{thm:A3}.
As a consequence of Theorem \ref{thm:KFqint}, we obtain the following theorem.
\begin{Theorem} \label{thm:A3}
Assume that the function $h(s)$ satisfies
\begin{equation*}
 A^{\langle 3 \rangle} \big( s; h'_1, h'_2, h'_3, l'_1, l'_2, l'_3, \alpha ', \beta ' \big) h(s) = E' h(s)
\end{equation*}
and
\begin{equation*}
 \lim _{L \to +\infty } \frac{h(s)}{s^{(h'_1 + h'_2 + h'_3 - l'_1 - l'_2 - l'_3 - 2 \alpha ' + \beta ' +2)/2}} \big| _{s=q^{L} \xi } = 0, \qquad \lim _{K \to -\infty } \frac{h(s)}{s^{-\alpha ' -1}} \big| _{s=q^{K} \xi } = 0
\end{equation*}
for some constant $E' $.
Then the Jackson integral
\begin{equation*}
 g (x) = x^{- \alpha } \int^{\xi \infty }_{0} s^{-(h'_1 + h'_2 + h'_3 - l'_1 - l'_2 - l'_3 - 2 \alpha ' + \beta ' +2)/2} h(s) P^{(1)}_{\mu ,\mu _0 } (x,s) {\rm d}_{q}s
\end{equation*}
converges and it satisfies
\begin{equation*}
 A^{\langle 3 \rangle} ( x; h_1, h_2, h_3, l_1, l_2, l_3, \alpha , \beta ) g (x) = E g(x) ,
\end{equation*}
where
\begin{gather*}
 \chi = \big( l'_1 + l'_2 + l'_3 - h'_1 - h'_2 - h'_3 - \beta ' \big)/2 , \qquad E = q^{2 \mu _0 + \alpha - \alpha ' + \chi} E', \qquad \mu = \mu _0 +1+ \chi, \\
 \beta = -\beta ' - \chi , \qquad l_i = l'_i +\mu _0 , \qquad h_i = h'_i +\mu _0 +\chi ,\qquad i=1,2,3.
\end{gather*}
\end{Theorem}

%\subsection{}

We investigate a $q$-integral transformation for the variant of the $q$-Heun equation of degree four.
The operator $A^{\langle 2 \rangle} ( x; h_1, h_2, h_3, h_4, l_1, l_2, l_3, l_4 ,\alpha ) $ in equation \eqref{eq:qsecondal} was used to write down the variant of the $q$-Heun equation of degree four.
In Theorem \ref{thm:A2KerId}, we obtained the identity
\begin{gather}
 A^{\langle 2 \rangle} ( x; h_1, h_2, h_3, h_4, l_1, l_2, l_3, l_4 ,\alpha ) x^{-\alpha } s^{\chi +1 - \tilde{\alpha }} P_{\mu ,\mu _0 } (x,s) \nonumber \\
\qquad = q^{\nu }\big\{ A^{\langle 2 \rangle} \big( s; \tilde{h}_1, \tilde{h}_2, \tilde{h}_3, \tilde{h}_4, \tilde{l}_1, \tilde{l}_2, \tilde{l}_3, \tilde{l}_4 , \tilde{\alpha } \big) \big\} x^{-\alpha } s^{\chi +1 - \tilde{\alpha }} P_{\mu ,\mu _0 } (x,s), \label{eq:A2KerId1}
\end{gather}
where the parameters satisfy equation \eqref{eq:A2KerIdp}.
We now apply Theorem \ref{thm:KFqint} for the variant of the $q$-Heun equation of degree four.
Let $E $ and $ \tilde{E} $ be constants such that $E = q^{\nu } \tilde{E} $ and write
\begin{align*}
& a(x) T_{x}^{-1} + b(x) +c(x) T_{x}= A^{\langle 2 \rangle} ( x; h_1, h_2, h_3, h_4, l_1, l_2, l_3, l_4 ,\alpha ) -E , \\
& \tilde{a}(s) T_{s}^{-1} + \tilde{b}(s) + \tilde{c}(s) T_{s} =q^{\nu } \big\{ A^{\langle 2 \rangle} \big( s; \tilde{h}_1, \tilde{h}_2, \tilde{h}_3, \tilde{h}_4, \tilde{l}_1, \tilde{l}_2, \tilde{l}_3, \tilde{l}_4 , \tilde{\alpha } \big) - \tilde{E} \big\} . \nonumber
\end{align*}
Equation \eqref{eq:A2KerId1} gives a realization of equation \eqref{eq:Phixs} by setting $ \Phi (x,s) = x^{-\alpha } s^{\chi +1 - \tilde{\alpha }} P_{\mu ,\mu _0 } (x,s)$.
In this situation, the coefficients of the equation
\begin{gather}
\big\{ q \tilde{a}(qs) T_{s} + \tilde{b}(s) + q^{-1} \tilde{c}(s/q) T_{s}^{-1} \big\} h(s) = 0 \label{eq:A2hs}
\end{gather}
are described as
\begin{gather*}
 q\tilde{a}(qs) = q^{\nu +3 } s^{-2} \big(s-q^{\tilde{h}_1 - 1/2} t_1\big) \big(s-q^{\tilde{h}_2 -1/2} t_2\big) \big(s-q^{\tilde{h}_3 -1/2} t_3\big) \big(s-q^{\tilde{h}_4 -1/2} t_4\big) , \\
 q^{-1} \tilde{c}(s/q) = q^{\nu + 2\tilde{\alpha } -2} s^{-2} \big(s - q^{\tilde{l}_1 +1/2} t_1\big) \big(s - q^{\tilde{l}_2 +1/2} t_2\big) \big(s - q^{\tilde{l}_3 +1/2} t_3\big) \big(s - q^{\tilde{l}_4 +1/2} t_4\big) , \nonumber \\
 \tilde{b}(s) = q^{\nu + \tilde{\alpha } +1/2} \Biggl[ -\big(q^{1/2} +q^{-1/2} \big) s^2 + \sum _{n=1}^4 \big(q^{\tilde{h}_n}+ q^{\tilde{l}_n} \big) t_n s \nonumber \\
 \phantom{ \tilde{b}(s) = }{}  + \prod_{n=1}^4 q^{(\tilde{h}_n +\tilde{l}_n)/2} t_n \cdot \left\{ - \big(q^{1/2} +q^{-1/2} \big) s^{-2} + \sum _{n=1}^4 \left( \frac{1}{q^{\tilde{h}_n}t_n} + \frac{1}{q^{\tilde{l}_n}t_n} \right) s^{-1} \right\} \Biggr] - q^{\nu } \tilde{E} . \nonumber
\end{gather*}
Equation \eqref{eq:A2hs} is equivalent to the equation $A^{\langle 2 \rangle} \big( x; h'_1, h'_2, h'_3, h'_4, l'_1, l'_2, l'_3, l'_4, \alpha ' \big) h(s) = E' h(s) $, where
\begin{equation*}
 l'_i =\tilde{h}_i , \qquad h'_i = \tilde{l}_i , \qquad i=1,2,3,4 ,\qquad \alpha ' = 2 - \tilde{\alpha } , \qquad E' = q^{2 - 2 \tilde{\alpha } } \tilde{E } .
\end{equation*}
Note that the exponents of equation \eqref{eq:A2hs} about $s=\infty $ are $\{ \alpha ' , \alpha ' + 1 \}$, the exponents about~${s=0}$ are $\{ \lambda ' , \lambda ' +1 \}$, where $\lambda ' = \big( h'_1 +h'_2 +h'_3 +h'_4 -l'_1-l'_2 -l'_3 -l'_4 - 2 \alpha ' +2\big)/2$, and the singular points $s=0 $ and $s=\infty $ are non-logarithmic.
The values $\lambda ' _+ + 1 \bigl( = \big( h'_1 +h'_2 +h'_3 +h'_4 -l'_1-l'_2 -l'_3 -l'_4 - 2 \alpha ' +4\big)/2 \bigr)$, $\lambda ' _+ $ and $\alpha ' +1$ will appear in Theorem~\ref{thm:A2}.
As a consequence of Theorem \ref{thm:KFqint}, we obtain the following theorem.
\begin{Theorem} \label{thm:A2}
Assume that the function $h(s)$ satisfies
\begin{equation*}
 A^{\langle 2 \rangle} \big( s; h'_1, h'_2, h'_3, h'_4, l'_1, l'_2, l'_3, l'_4, \alpha ' \big) h(s) = E' h(s)
\end{equation*}
and
\begin{equation*}
 \lim _{L \to +\infty } \frac{h(s)}{s^{( h'_1 +h'_2 +h'_3 +h'_4 -l'_1-l'_2 -l'_3 -l'_4 - 2 \alpha ' +4)/2}} \big| _{s=q^{L} \xi } = 0, \qquad \lim _{K \to -\infty } \frac{h(s)}{s^{-\alpha ' -1}} \big| _{s=q^{K} \xi } = 0
\end{equation*}
for some constant $E'$.
Then the Jackson integral
\begin{equation*}
 g (x) = x^{- \alpha } \int^{\xi \infty }_{0} s^{-( h'_1 +h'_2 +h'_3 +h'_4 -l'_1-l'_2 -l'_3 -l'_4 - 2 \alpha ' +2)/2} h(s) P^{(1)}_{\mu ,\mu _0 } (x,s) {\rm d}_{q}s
\end{equation*}
converges and it satisfies
\begin{equation*}
 A^{\langle 2 \rangle} ( x; h_1, h_2, h_3, h_4, l_1, l_2, l_3, l_4 , \alpha ) g (x) = E g(x) ,
\end{equation*}
where
\begin{gather*}
 \chi = \big( l'_1 + l'_2 + l'_3 + l'_4 - h'_1 - h'_2 - h'_3 - h'_4 \big)/2 , \qquad E = q^{2 \mu _0 + \alpha - \alpha ' + \chi} E', \\
 \mu = \mu _0 +1+ \chi, \qquad l_i = l'_i +\mu _0 , \qquad h_i = h'_i +\mu _0 +\chi , \qquad i=1,2,3,4. \nonumber
\end{gather*}
\end{Theorem}

\section[Special solutions of the q-Heun equation]{Special solutions of the $\boldsymbol{q}$-Heun equation} \label{sec:solqHeun}

We investigate special solutions of the $q$-Heun equation.
Recall that the $q$-Heun equation is written as
\begin{equation*}
A^{\langle 4 \rangle} ( x; h_1, h_2, l_1, l_2, \alpha _1, \alpha _2, \beta ) g(x) =Eg(x),
\end{equation*}
where $E$ is an arbitrary complex number and $A^{\langle 4 \rangle} ( x; h_1, h_2, l_1, l_2, \alpha _1, \alpha _2, \beta )$ is written as
\begin{gather*}
 A^{\langle 4 \rangle} ( x; h_1, h_2, l_1, l_2, \alpha _1, \alpha _2, \beta ) \\
 \qquad= x^{-1} \big(x-q^{h_1 + 1/2} t_1\big) \big(x-q^{h_2 +1/2} t_2\big) T_{x}^{-1} \\
\phantom{ \qquad=}{}+ q^{\alpha _1 +\alpha _2} x^{-1} \big(x - q^{l_1 -1/2} t_1\big) \big(x - q^{l_2 -1/2} t_2\big) T_{x} \\
\phantom{ \qquad=}{} -\big\{ \big(q^{\alpha _1} +q^{\alpha _2} \big) x + q^{(h_1 +h_2 + l_1 +l_2 +\alpha _1 +\alpha _2 )/2} \big( q^{\beta/2} + q^{-\beta/2} \big) t_1 t_2 x^{-1} \big\} .
\end{gather*}

%\subsection{}

To investigate monomial solutions of the $q$-Heun equation, we look for monomial eigenfunctions of the operator $A^{\langle 4 \rangle} ( x; h_1, h_2, l_1, l_2, \alpha _1, \alpha _2, \beta )$.
Although the following proposition is essentially a special case of \cite[Proposition 4.1]{TakqH}, we reformulate it and give a proof.
\begin{Proposition}[{cf.\ \cite[Proposition 4.1]{TakqH}}] \label{prop:monom}
Set $(i,i')=(1,2) $ or $(2,1)$.
If $\pm \beta = h_1 +h_2 - l_1 -l_2 +\alpha _i - \alpha _{i' } +2$, then
\begin{gather*}
 A^{\langle 4 \rangle} ( x; h_1, h_2, l_1, l_2, \alpha _1, \alpha _2, \beta ) x ^{-\alpha _i } \\
\qquad = -\big( q^{\alpha _i + h_1 + 1/2} t_1 + q^{\alpha _i + h_2 + 1/2} t_2 + q^{\alpha _{i' } +l_1 -1/2} t_1 + q^{\alpha _{i' } +l_2 -1/2} t_2 \big) x ^{-\alpha _i} .
\end{gather*}
\end{Proposition}
\begin{proof}
Let us find eigenfunctions of the operator $A^{\langle 4 \rangle} $ in the form $x^{\nu }$.
We have
\begin{gather*}
\begin{split}
& A^{\langle 4 \rangle} ( x; h_1, h_2, l_1, l_2, \alpha _1, \alpha _2, \beta ) x^{\nu } \\
& \qquad = \big(q^{-\nu } + q^{\alpha _1 +\alpha _2 +\nu } -q^{\alpha _1} -q^{\alpha _2} \big) x^{\nu +1} + \big( q^{h_1 + h_2 + 1 - \nu } + q^{\alpha _1 +\alpha _2 + l_1 + l_2 - 1 + \nu } \\
& \phantom{\qquad =}{} - q^{(h_1 +h_2 + l_1 +l_2 +\alpha _1 +\alpha _2 +\beta )/2} - q^{(h_1 +h_2 + l_1 +l_2 +\alpha _1 +\alpha _2 -\beta )/2} \big) t_1 t_2 x^{\nu -1} \\
& \phantom{\qquad =}{} -\big( q^{-\nu + h_1 + 1/2} t_1 + q^{-\nu + h_2 + 1/2} t_2 + q^{\alpha _1 + \alpha _2 + \nu + l_1 -1/2} t_1 + q^{\alpha _1 + \alpha _2 + \nu +l_2 -1/2} t_2 \big) x^{\nu } .
\end{split}
\end{gather*}
If $\nu = -\alpha _1$ or $\nu = -\alpha _2$, then the coefficient of $x^{\nu +1}$ vanishes.
In the case $\nu = -\alpha _1$ (resp.\ ${\nu = -\alpha _2 }$), the condition $\pm \beta = h_1 +h_2 - l_1 -l_2 +\alpha _1 - \alpha _2 +2 $ (resp.~$\pm \beta = h_1 +h_2 - l_1 -l_2 -\alpha _1 + \alpha _2 +2$) is sufficient for elimination of the term $x^{\nu -1} $.
Then we obtain the expression in the proposition.
\end{proof}

The monomial $ x ^{-\alpha _i} $ satisfies $\big[T_{x} ^{-1} -q^{\alpha _{i }} \big] x ^{-\alpha _i} =0 $, and existence of the monomial solution is related to the factorization of a $q$-difference operator.
Namely, if $\pm \beta = h_1 +h_2 - l_1 -l_2 +\alpha _i - \alpha _{i' } +2$, then we have
\begin{gather*}
 x^{-1} \big[ \big(x-q^{h_1 + 1/2} t_1\big) \big(x-q^{h_2 +1/2} t_2\big) - q^{\alpha _{i' }} \big(x - q^{l_1 -1/2} t_1\big) \big(x - q^{l_2 -1/2} t_2\big) T_{x} \big] \big[T_{x} ^{-1} -q^{\alpha _{i }}\big] \\
\qquad = A^{\langle 4 \rangle} ( x; h_1, h_2, l_1, l_2, \alpha _1, \alpha _2, \beta ) \\
 \phantom{\qquad = }{}+ q^{\alpha _i + h_1 + 1/2} t_1 + q^{\alpha _i + h_2 + 1/2} t_2 + q^{\alpha _{i' } +l_1 -1/2} t_1 + q^{\alpha _{i' } +l_2 -1/2} t_2 .
\end{gather*}

On a gauge transformation, the following proposition is shown readily.
\begin{Proposition}
If the function $f(x) $ satisfies
\begin{gather*}
 x^{-1} \big(x-q^{l_1 + 1/2} t_1\big) \big(x-q^{h_2 +1/2} t_2\big) f(x/q) + q^{\alpha _1 +\alpha _2} x^{-1} \big(x - q^{h_1 -1/2} t_1\big) \big(x - q^{l_2 -1/2} t_2\big) f(qx) \\
 \qquad-\big\{ \big(q^{\alpha _1} +q^{\alpha _2} \big) x + q^{(h_1 +h_2 + l_1 +l_2 +\alpha _1 +\alpha _2 )/2} \big( q^{\beta/2} + q^{-\beta/2} \big) t_1 t_2 x^{-1} \big\} f(x) = E f(x)
\end{gather*}
and
\begin{equation*}
 g_1(x) = \frac{\big(q^{h_1 + 1/2} t_1/x; q\big)_{\infty }}{\big(q^{l_1 + 1/2} t_1/x; q\big)_{\infty }} f(x) , \qquad g_2 (x) = x ^{h_1 -l_1} \frac{\big(x/\big(q^{l_1 - 1/2} t_1\big); q\big)_{\infty }}{(x/(q^{h_1 - 1/2} t_1) ; q)_{\infty }} f(x),
\end{equation*}
then the functions $g_1 (x)$ and $g_2 (x)$ satisfy
\begin{gather*}
 x^{-1} \big(x-q^{h_1 + 1/2} t_1\big) \big(x-q^{h_2 +1/2} t_2\big) g(x/q) + q^{\alpha _1 +\alpha _2} x^{-1} \big(x - q^{l_1 -1/2} t_1\big) \big(x - q^{l_2 -1/2} t_2\big) g(qx) \\
 \qquad-\big\{ \big(q^{\alpha _1} +q^{\alpha _2} \big) x + q^{(h_1 +h_2 + l_1 +l_2 +\alpha _1 +\alpha _2 )/2} \big( q^{\beta/2} + q^{-\beta/2} \big) t_1 t_2 x^{-1} \big\} g(x) = E g(x) .
\end{gather*}
\end{Proposition}
By applying gauge transformations to Proposition \ref{prop:monom}, we have the following solutions of the $q$-Heun equation.
\begin{Proposition} \label{prop:monomGT}
Set $(i,i')=(1,2) $ or $(2,1)$.
\begin{itemize}\itemsep=0pt
\item[$(i)$] If $\pm \beta =- h_1 +h_2 + l_1 -l_2 +\alpha _i - \alpha _{i' } +2$, then
\begin{gather*}
 A^{\langle 4 \rangle} ( x; h_1, h_2, l_1, l_2, \alpha _1, \alpha _2, \beta ) x ^{-\alpha _i} \frac{\big(q^{h_1 + 1/2} t_1/x; q\big)_{\infty }}{\big(q^{l_1 + 1/2} t_1/x; q\big)_{\infty }} \\
\qquad = -\big( q^{\alpha _i + l_1 + 1/2} t_1 + q^{\alpha _i + h_2 +1/2} t_2 + q^{\alpha _{i' } + h_1 - 1/2} t_1 + q^{\alpha _{i' } + l_2 -1/2} t_2\big) \\
 \phantom{\qquad =}{} \times x ^{-\alpha _i} \frac{\big(q^{h_1 + 1/2} t_1/x; q\big)_{\infty }}{\big(q^{l_1 + 1/2} t_1/x; q\big)_{\infty }}, \\
 A^{\langle 4 \rangle} ( x; h_1, h_2, l_1, l_2, \alpha _1, \alpha _2, \beta ) x ^{-\alpha _i +h_1 -l_1} \frac{(x/\big(q^{l_1 - 1/2} t_1\big); q)_{\infty }}{\big(x/\big(q^{h_1 - 1/2} t_1\big) ; q\big)_{\infty }} \\
\qquad= -\big( q^{\alpha _i + l_1 + 1/2} t_1 + q^{\alpha _i + h_2 +1/2} t_2 + q^{\alpha _{i' } + h_1 - 1/2} t_1 + q^{\alpha _{i' } + l_2 -1/2} t_2\big) \\
 \phantom{\qquad =}{} \times x ^{-\alpha _i +h_1 -l_1} \frac{\big(x/\big(q^{l_1 - 1/2} t_1\big); q\big)_{\infty }}{\big(x/\big(q^{h_1 - 1/2} t_1\big) ; q\big)_{\infty }} .
\end{gather*}
\item[$(ii)$] If $\pm \beta =- h_1 - h_2 + l_1 + l_2 +\alpha _i - \alpha _{i' } +2$, then
\begin{gather*}
 A^{\langle 4 \rangle} ( x; h_1, h_2, l_1, l_2, \alpha _1, \alpha _2, \beta ) x ^{-\alpha _i} \frac{\big(q^{h_1 + 1/2} t_1/x, q^{h_2 + 1/2} t_2/x; q\big)_{\infty }}{\big(q^{l_1 + 1/2} t_1/x, q^{l_2 + 1/2} t_2/x; q\big)_{\infty }} \\
\qquad = -\big( q^{\alpha _i + l_1 + 1/2} t_1 + q^{\alpha _i + l_2 +1/2} t_2 + q^{\alpha _{i' } + h_1 - 1/2} t_1 + q^{\alpha _{i' } + h_2 -1/2} t_2\big) \\
 \phantom{\qquad =}{} \times x ^{-\alpha _i} \frac{\big(q^{h_1 + 1/2} t_1/x, q^{h_2 + 1/2} t_2/x; q\big)_{\infty }}{\big(q^{l_1 + 1/2} t_1/x, q^{l_2 + 1/2} t_2/x; q\big)_{\infty }}, \\
 A^{\langle 4 \rangle} ( x; h_1, h_2, l_1, l_2, \alpha _1, \alpha _2, \beta ) x ^{-\alpha _i +h_1+h_2 -l_1 -l_2} \frac{\big(x/\big(q^{l_1 - 1/2} t_1\big), x/\big(q^{l_2 - 1/2} t_2\big); q\big)_{\infty }}{\big(x/\big(q^{h_1 - 1/2} t_1\big) , x/\big(q^{h_2 - 1/2} t_2\big) ; q\big)_{\infty }} \\
 \qquad= -\big( q^{\alpha _i + l_1 + 1/2} t_1 + q^{\alpha _i + l_2 +1/2} t_2 + q^{\alpha _{i' } + h_1 - 1/2} t_1 + q^{\alpha _{i' } + h_2 -1/2} t_2\big) \\
 \phantom{\qquad =}{} \times x ^{-\alpha _i +h_1+h_2 -l_1 -l_2} \frac{\big(x/\big(q^{l_1 - 1/2} t_1\big), x/\big(q^{l_2 - 1/2} t_2\big); q\big)_{\infty }}{\big(x/\big(q^{h_1 - 1/2} t_1\big) , x/\big(q^{h_2 - 1/2} t_2\big) ; q\big)_{\infty }} .
\end{gather*}
\end{itemize}
\end{Proposition}
Proposition \ref{prop:monomGT} is related to the factorization of a $q$-difference operator.
Namely, if $\pm \beta =- h_1 +h_2 + l_1 -l_2 +\alpha _i - \alpha _{i' } +2$, then
\begin{gather*}
 x^{-1} \big[ \big(x-q^{h_2 +1/2} t_2\big) - q^{\alpha_ {i' } -1} \big(x - q^{l_2 -1/2} t_2\big) T_{x} \big] \big[ \big(x-q^{h_1 + 1/2} t_1\big) T_{x} ^{-1} -q^{\alpha _i} \big(x - q^{l_1 + 1/2} t_1\big) \big] \\
\qquad = A^{\langle 4 \rangle} ( x; h_1, h_2, l_1, l_2, \alpha _1, \alpha _2, \beta ) \nonumber \\
 \phantom{\qquad =}{} + q^{\alpha _i + l_1 + 1/2} t_1 + q^{\alpha _i + h_2 +1/2} t_2 + q^{\alpha _{i' } + h_1 - 1/2} t_1 + q^{\alpha _{i' } + l_2 -1/2} t_2 ,
\end{gather*}
and, if $\pm \beta =- h_1 - h_2 + l_1 + l_2 +\alpha _i - \alpha _{i' } +2$, then
\begin{gather*}
 x^{-1} \big[ 1 - q^{\alpha _{i' } -2} T_{x} \big] \big[ \big(x-q^{h_1 + 1/2} t_1\big) \big(x-q^{h_2 + 1/2} t_2\big) T_{x} ^{-1} -q^{\alpha _i} \big(x - q^{l_1 + 1/2} t_1\big) \big(x - q^{l_2 + 1/2} t_2\big) \big] \\
\qquad = A^{\langle 4 \rangle} ( x; h_1, h_2, l_1, l_2, \alpha _1, \alpha _2, \beta ) \\
 \phantom{\qquad =}{} + q^{\alpha _i + l_1 + 1/2} t_1 + q^{\alpha _i + l_2 +1/2} t_2 + q^{\alpha _{i' } + h_1 - 1/2} t_1 + q^{\alpha _{i' } + h_2 -1/2} t_2 .
\end{gather*}

%\subsection{}
%In these three subsections,
We apply the $q$-integral transformations in Theorem \ref{thm:A4} to some solutions of the $q$-Heun equation given in Propositions \ref{prop:monomGT} and \ref{prop:monom}.

Assume that the parameters in $A^{\langle 4 \rangle} \big( x; h'_1, h'_2, l'_1, l'_2, \alpha '_1, \alpha '_2, \beta '\big) $ satisfy
\begin{gather}
 \beta ' =- h'_1 +h'_2 + l'_1 -l'_2 -\alpha '_1 + \alpha '_2 +2. \label{eq:beta'1}
\end{gather}
Then it follows from Proposition \ref{prop:monomGT}\,(i) that
\begin{gather*}
 A^{\langle 4 \rangle} \big( x; h'_1, h'_2, l'_1, l'_2, \alpha '_1, \alpha '_2, \beta '\big) x ^{-\alpha '_2 +h'_1 -l'_1} \frac{\big(x/\big(q^{l'_1 - 1/2} t_1\big); q\big)_{\infty }}{\big(x/\big(q^{h'_1 - 1/2} t_1\big) ; q\big)_{\infty }} \\
\qquad = -\big( q^{\alpha '_2 + l'_1 + 1/2} t_1 + q^{\alpha '_2 + h'_2 +1/2} t_2 + q^{\alpha '_1 + h'_1 - 1/2} t_1 + q^{\alpha '_1 + l'_2 -1/2} t_2\big) \\
\phantom{\qquad =}{} \times x ^{-\alpha '_2 +h'_1 -l'_1} \frac{\big(x/\big(q^{l'_1 - 1/2} t_1\big); q\big)_{\infty }}{\big(x/\big(q^{h'_1 - 1/2} t_1\big) ; q\big)_{\infty }} .
\end{gather*}
We apply Theorem \ref{thm:A4} in the case
\begin{gather*}
 E' = -\big( q^{\alpha '_2 + l'_1 + 1/2} t_1 + q^{\alpha '_2 + h'_2 +1/2} t_2 + q^{\alpha '_1 + h'_1 - 1/2} t_1 + q^{\alpha '_1 + l'_2 -1/2} t_2\big) , \\
 h(s) = s ^{-\alpha '_2 +h'_1 -l'_1} \frac{\big(s/\big(q^{l'_1 - 1/2} t_1\big); q\big)_{\infty }}{\big(s/\big(q^{h'_1 - 1/2} t_1\big) ; q\big)_{\infty }} , \qquad \mu _0 = 0 .
\end{gather*}
It follows from equation \eqref{eq:beta'1} that the parameters in Theorem \ref{thm:A4} satisfy
\begin{gather}
 \chi = - h'_2 + l'_2 -1 , \qquad \mu = - h'_2 + l'_2 , \qquad \alpha _2 = \alpha _1 - \alpha ' _1 + \alpha ' _2 + h'_2 - l'_2 + 1 , \nonumber \\
 \beta = h'_1 - l'_1 + \alpha '_1 - \alpha '_2 - 1, \qquad l_1 = l'_1 , \qquad l_2 = l'_2 , \qquad h_1 = h'_1 - h'_2 + l'_2 -1, \nonumber \\
 h_2 = l'_2 -1 ,\nonumber\\
 E = - q^{\alpha _1 }\big( q^{\alpha '_2 - \alpha ' _1+ l'_1 + 1/2} t_1 + q^{\alpha '_2 - \alpha ' _1 + h'_2 +1/2} t_2 + q^{h'_1 - 1/2} t_1 + q^{ l'_2 -1/2} t_2\big). \label{eq:paramA41}
\end{gather}
In particular, we have $h_2 = l_2 - 1$.
The Jackson integral $g(x)$ in equation \eqref{eq:A4Jint} is written as
\begin{gather}
 g (x) = x^{- \alpha _1} \int^{\xi \infty }_{0} s^{- \beta ' } \frac{\big(s/\big(q^{l'_1 - 1/2} t_1\big); q\big)_{\infty }}{\big(s/\big(q^{h'_1 - 1/2} t_1\big) ; q\big)_{\infty }} \frac{\big(q^{ - h'_2 + l'_2}s/x;q\big)_{\infty }}{(s/x;q)_{\infty }} {\rm d}_{q}s \nonumber \\
\qquad = (1-q) \xi ^{-\beta ' +1} x^{- \alpha _1} \frac{\big(q^{ - l'_1 + 1/2} \xi / t_1, q^{ - h'_2 + l'_2} \xi /x ; q\big)_{\infty }}{\big(q^{ - h'_1 + 1/2} \xi / t_1 , \xi /x; q\big)_{\infty }} \nonumber \\
 \phantom{\qquad =}{} \times \sum ^{\infty }_{n=-\infty } \frac{\big(q^{ - h'_1 + 1/2} \xi / t_1,\xi /x ; q\big)_{n }}{\big(q^{- l'_1 + 1/2} \xi / t_1,q^{- h'_2 + l'_2} \xi /x ; q\big)_{n }} q^{(-\beta ' +1 )n} . \label{eq:gxA41}
\end{gather}
If the parameters satisfy $\beta ' <0$ and $\alpha '_1<\alpha '_2 $, then the constants $C_1$ and $C_2$ in Theorem \ref{thm:A4} are equal to $0$ respectively, the Jackson integral $g(x)$ in equation \eqref{eq:gxA41} converges and it satisfies
\begin{gather*}
 A^{\langle 4 \rangle} ( x; h_1, h_2, l_1, l_2, \alpha _1, \alpha _2, \beta ) g (x) = E g(x) ,
\end{gather*}
where the parameters satisfy equation \eqref{eq:paramA41}.
We have
\begin{gather*}
 A^{\langle 4 \rangle} ( x; h_1, h_2, l_1, l_2, \alpha _1, \alpha _2, \beta ) \\
 \qquad{}+ q^{\alpha _1 }\big( q^{\alpha '_2 - \alpha ' _1+ l'_1 + 1/2} t_1 + q^{\alpha '_2 - \alpha ' _1 + h'_2 +1/2} t_2 + q^{h'_1 - 1/2} t_1 + q^{ l'_2 -1/2} t_2\big) \\
\quad= x^{-1} \big(x-q^{ l'_2 -1/2} t_2\big) \big[ \big(x-q^{h'_1 - h'_2 + l'_2 -1/2} t_1\big)T_{x}^{-1} + q^{2 \alpha _1 - \alpha ' _1 + \alpha ' _2 + h'_2 - l'_2 + 1} \big(x - q^{l'_1 -1/2} t_1\big) T_{x} \\
 \phantom{\quad =}{}-q^{\alpha _1} \big\{ \big(1 +q^{ - \alpha ' _1 + \alpha ' _2 + h'_2 - l'_2 + 1 }\big) x - \big(q^{h'_1 - 1/2} + q^{- \alpha ' _1 + \alpha '_2 + l'_1 + 1/2} \big) t_1\big\} \big] ,
\end{gather*}
and we essentially obtain the $q$-hypergeometric equation.
If $\xi = q^{l'_1+1/2} t_1 $, then
\begin{align*}
 g (x) = {}&(1-q) \big(q^{l'_1+1/2} t_1\big) ^{-\beta ' +1} x^{- \alpha _1} \frac{\big(q , q^{ - h'_2 + l'_2 + l'_1+1/2} t_1 /x;q\big)_{\infty } }{\big(q^{ - h'_1 + l'_1+1} , q^{l'_1+1/2} t_1 /x;q\big)_{\infty } } \\
& \times {}_2\phi_1 \biggl( \begin{array}{@{}c@{}} q^{l'_1+1/2} t_1 /x, q^{ - h'_1 +l'_1+1} \\
q^{- h'_2 + l'_2 + l'_1+1/2} t_1 /x \end{array} ;q,q^{-\beta ' +1} \biggr) . \nonumber
\end{align*}
If $\xi = q^{ h'_2 - l'_2 +1} x $, then
\begin{align*}
 g (x) ={}& (1-q) \big( q^{ h'_2 - l'_2 +1} x \big) ^{-\beta ' +1} x^{- \alpha _1} \frac{\big(q^{ h'_2 - l'_1 - l'_2 +3/2} x / t_1,q ; q\big)_{\infty } }{\big(q^{ - h'_1+ h'_2 - l'_2 +3/2} x / t_1, q^{ h'_2 - l'_2 +1} ; q\big)_{\infty }} \\
& \times {}_2\phi_1 \biggl( \begin{array}{@{}c@{}}q^{ - h'_1+ h'_2 - l'_2 +3/2} x / t_1, q^{ h'_2 - l'_2 +1} \\
 q^{ h'_2 - l'_1 - l'_2 +3/2} x / t_1 \end{array} ;q,q^{-\beta ' +1} \biggr) . \nonumber
\end{align*}
These results agree with the $q$-integral representations of the $q$-hypergeometric equation obtained in \cite{AT}.

%\subsection{}

Assume that the parameters in $A^{\langle 4 \rangle} \big( x; h'_1, h'_2, l'_1, l'_2, \alpha '_1, \alpha '_2, \beta '\big) $ satisfy
\begin{gather}
 \beta ' =- h'_1 - h'_2 + l'_1 + l'_2 -\alpha '_1 + \alpha '_2 +2. \label{eq:beta'2}
\end{gather}
Then it follows from Proposition \ref{prop:monomGT}\,(ii) that
\begin{gather*}
 A^{\langle 4 \rangle} \big( x; h'_1, h'_2, l'_1, l'_2, \alpha '_1, \alpha '_2, \beta ' \big) x ^{-\alpha '_2 +h'_1+h'_2 -l'_1 -l'_2} \frac{\big(x/\big(q^{l'_1 - 1/2} t_1, x/\big(q^{l'_2 - 1/2} t_2\big); q\big)_{\infty }}{\big(x/\big(q^{h'_1 - 1/2} t_1\big) , x/\big(q^{h'_2 - 1/2} t_2\big) ; q\big)_{\infty }} \\
\qquad = -\big( q^{\alpha '_2 + l'_1 + 1/2} t_1 + q^{\alpha '_2 + l'_2 +1/2} t_2 + q^{\alpha '_1 + h'_1 - 1/2} t_1 + q^{\alpha '_1 + h'_2 -1/2} t_2\big) \\
 \phantom{\qquad =}{}\times x ^{-\alpha '_2 +h'_1+h'_2 -l'_1 -l'_2} \frac{\big(x/\big(q^{l'_1 - 1/2} t_1\big), x/\big(q^{l'_2 - 1/2} t_2\big); q\big)_{\infty }}{\big(x/\big(q^{h'_1 - 1/2} t_1\big) , x/\big(q^{h'_2 - 1/2} t_2\big) ; q\big)_{\infty }} .
\end{gather*}
We apply Theorem \ref{thm:A4} in the case
\begin{gather*}
 E' = -\big( q^{\alpha '_2 + l'_1 + 1/2} t_1 + q^{\alpha '_2 + l'_2 +1/2} t_2 + q^{\alpha '_1 + h'_1 - 1/2} t_1 + q^{\alpha '_1 + h'_2 -1/2} t_2\big) , \\
 h(s) = s ^{-\alpha '_2 +h'_1+h'_2 -l'_1 -l'_2} \frac{\big(s/\big(q^{l'_1 - 1/2} t_1\big), s/\big(q^{l'_2 - 1/2} t_2\big); q\big)_{\infty }}{\big(s/\big(q^{h'_1 - 1/2} t_1\big) , s/\big(q^{h'_2 - 1/2} t_2\big) ; q\big)_{\infty }} , \qquad \mu _0 = 0 .
\end{gather*}
It follows from equation \eqref{eq:beta'2} that the parameters in Theorem \ref{thm:A4} satisfy
\begin{gather}
\chi = - \beta ' -1 = (- l_1 - l_2 + h_1 + h_2 + \alpha _1- \alpha _2 -1)/2 , \qquad \mu = \chi +1 , \nonumber\\
 \beta = 1, \qquad \alpha _2 = \alpha _1 - \alpha ' _1 + \alpha ' _2 - \chi , \qquad l_1 = l'_1 , \qquad l_2 = l'_2 , \qquad h_1 = h'_1 +\chi , \nonumber \\
 h_2 = h'_2 +\chi , \nonumber \\
 E = -q^{\alpha _1}\big( q^{\alpha '_2 - \alpha ' _1 + l'_1 + 1/2} t_1 + q^{\alpha '_2 - \alpha ' _1 + l'_2 +1/2} t_2 + q^{h'_1 - 1/2} t_1 + q^{ h'_2 -1/2} t_2\big) . \label{eq:paramA42}
\end{gather}
The Jackson integral $g(x)$ in equation \eqref{eq:A4Jint} is written as
\begin{align}
 g (x) ={}&x^{- \alpha _1} \int^{\xi \infty }_{0} \frac{\big(s/\big(q^{l'_1 - 1/2} t_1\big), s/\big(q^{l'_2 - 1/2} t_2\big); q\big)_{\infty }}{\big(s/\big(q^{h'_1 - 1/2} t_1\big) , s/\big(q^{h'_2 - 1/2} t_2\big) ; q\big)_{\infty }} \frac{\big(q^{\mu }s/x;q\big)_{\infty }}{(s/x;q)_{\infty }} {\rm d}_{q}s \nonumber \\
={}& (1-q) \xi x^{- \alpha _1} \frac{\big( q^{-l_1 +1/2} \xi / t_1 , q^{ -l_2 +1/2 } \xi /t_2 , q^{ \chi +1 } \xi /x ; q\big)_{\infty }}{\big(q^{ - h_1 + \chi + 1/2} \xi / t_1 , q^{ -h_2 + \chi + 1/2} \xi / t_2 , \xi /x;q\big)_{\infty }} \nonumber \\
& \times \sum ^{\infty }_{n=-\infty } \frac{\big(q^{ - h_1 + \chi + 1/2} \xi / t_1 , q^{ -h_2 + \chi + 1/2} \xi / t_2 , \xi /x;q\big)_{n }}{\big( q^{ -l_1 +1/2} \xi / t_1 , q^{ -l_2 +1/2 } \xi /t_2 , q^{\chi +1 } \xi /x ; q\big)_{n }} q^{n} . \label{eq:gxA42}
\end{align}
If the parameters satisfy $\alpha '_1<\alpha '_2 $, then the constants $C_1$ and $C_2$ in Theorem \ref{thm:A4} satisfy $C_1 =1$ and $C_2 =0$, the Jackson integral $g(x)$ in equation \eqref{eq:gxA42} converges and it satisfies
\begin{gather}
 A^{\langle 4 \rangle} ( x; h_1, h_2, l_1, l_2, \alpha _1, \alpha _2, \beta ) g (x)\nonumber\\
 \qquad= E g(x) - (1-q) x^{- \alpha _1} q^{ \alpha _1 + h_1 + h_2 - \chi } \big( q^{ - \chi -1} - 1 \big) t_1 t_2 , \label{eq:thmA4gEgg1}
\end{gather}
where the parameters satisfy equation \eqref{eq:paramA42}.
Note that the homogeneous version of equation~\eqref{eq:thmA4gEgg1} is written in the variant of the $q$-hypergeometric equation of degree two (see \cite{HMST} for the definition), and our results agree with the results obtained in \cite{AT}.

%\subsection{}

Assume that the parameters in $A^{\langle 4 \rangle} \big( x; h'_1, h'_2, l'_1, l'_2, \alpha '_1, \alpha '_2, \beta '\big) $ satisfy
\begin{gather}
 \beta ' =h'_1 +h'_2 - l'_1 -l'_2 -\alpha '_1 + \alpha '_2 +2. \label{eq:beta'3}
\end{gather}
Then it follows from Proposition \ref{prop:monom} that
\begin{gather*}
 A^{\langle 4 \rangle} \big( x; h'_1, h'_2, l'_1, l'_2, \alpha '_1, \alpha '_2, \beta ' \big) x ^{-\alpha '_2} \\
\qquad = -\big( q^{\alpha '_2 + h'_1 + 1/2} t_1 + q^{\alpha ' _2 + h'_2 + 1/2} t_2 + q^{\alpha '_1 +l'_1 -1/2} t_1 + q^{\alpha '_1 +l'_2 -1/2} t_2\big) x ^{-\alpha '_2} .
\end{gather*}
We apply Theorem \ref{thm:A4} in the case
\begin{gather*}
 E' = -\big( q^{\alpha '_2 + h'_1 + 1/2} t_1 + q^{\alpha ' _2 + h'_2 + 1/2} t_2 + q^{\alpha '_1 +l'_1 -1/2} t_1 + q^{\alpha '_1 +l'_2 -1/2} t_2\big) , \\
 h(s) = s ^{-\alpha '_2} , \qquad \mu _0 = 0 .
\end{gather*}
It follows from equation \eqref{eq:beta'3} that the parameters in Theorem \ref{thm:A4} satisfy
\begin{gather}
 \chi = - \beta ' -\alpha '_1 + \alpha '_2 +1 , \qquad \mu = \chi +1 , \qquad \alpha _2 = \alpha _1 + \beta ' - 1, \nonumber \\
 \beta = \alpha '_1 - \alpha '_2 -1 , \qquad l_1 = l'_1 ,\qquad l_2 = l'_2 , \qquad h_1 = h'_1 +\chi ,\qquad h_2 = h'_2 +\chi , \nonumber \\
 E = -q^{\alpha _1}\big(q^{\alpha '_2 - \alpha ' _1 + h'_1 + 1/2} t_1 + q^{\alpha ' _2 - \alpha ' _1 + h'_2 + 1/2} t_2 + q^{\alpha '_1 +l'_1 -1/2} t_1 + q^{\alpha '_1 +l'_2 -1/2} t_2\big) . \label{eq:paramA43}
\end{gather}
In particular, we have $ \beta + h_1 +h_2 -l_1 -l_2 -\alpha _1 +\alpha _2 +2=0$.
The Jackson integral $g(x)$ in equation \eqref{eq:A4Jint} is written as
\begin{align}
 g (x) ={}& x^{- \alpha _1} \int^{\xi \infty }_{0} s^{-\beta ' } \frac{\big(q^{- \beta ' -\alpha '_1 + \alpha '_2 +2}s/x;q\big)_{\infty }}{(s/x;q)_{\infty }} {\rm d}_{q}s \nonumber \\
={}& (1-q) \xi ^{-\beta ' +1} x^{- \alpha _1} \frac{\big(q^{- \beta ' -\alpha '_1 + \alpha '_2 +2 } \xi /x;q\big)_{\infty }}{(\xi /x;q)_{\infty }}\nonumber \\
&\times\sum ^{\infty }_{n=-\infty } \frac{(\xi /x;q)_{n }}{\big(q^{- \beta ' -\alpha '_1 + \alpha '_2 +2 } \xi /x;q\big)_{n }} q^{(-\beta ' +1 )n} . \label{eq:gxA43}
\end{align}
If the parameters satisfy $\beta ' <0$ and $\alpha '_1<\alpha '_2 $, then the constants $C_1$ and $C_2$ in Theorem \ref{thm:A4} are equal to $0$ respectively, the Jackson integral $g(x)$ in equation \eqref{eq:gxA43} converges and it satisfies
\begin{equation*}
 A^{\langle 4 \rangle} ( x; h_1, h_2, l_1, l_2, \alpha _1, \alpha _2, \beta ) g (x) = E g(x) ,
\end{equation*}
where the parameters satisfy equation \eqref{eq:paramA43}.

Recall that the Ramanujan's sum for $\! _1 \psi _1 (a; b ;q,z) $ (the bilateral summation formula) is written as
\begin{align}
& \sum _{n=-\infty }^{\infty } \frac{(a;q)_n}{(b;q)_n} z^n = \frac{(q,b/a,az,q/(az);q)_{\infty }}{(b,q/a,z,b/(az);q)_{\infty }} \label{eq:Ram}
\end{align}
for $|b/a|<|z|<1$ and $ |q|<1$.
It follows from equation \eqref{eq:Ram} in the case $a= \xi /x$, ${b= q^{- \beta ' -\alpha '_1 + \alpha '_2 +2 } \xi /x}$, $z= q^{-\beta ' +1 } $ that the function $g (x)$ in equation \eqref{eq:gxA43} is written as
\begin{gather}
 g(x) = x^{- \alpha _1} \frac{\big( q^{-\beta ' +1 } \xi /x , x /\big( q^{-\beta ' } \xi \big) , q^{- \beta ' -\alpha '_1 + \alpha '_2 +2 } , q ;q\big)_{\infty }}{\big(\xi /x, q x/ \xi , q^{-\beta ' +1 } , q^{-\alpha '_1 + \alpha '_2 +1 } ;q\big)_{\infty }} . \label{eq:gxA431}
\end{gather}
Since $\vartheta _q (q a x)/ \vartheta _q (q b x) = (b/a) \vartheta _q (a x)/ \vartheta _q (b x)$, the function $ g (x) $ in equation \eqref{eq:gxA431} behaves the same as $ x^{- \alpha _2}$ on the transformation $x \mapsto qx$.
Therefore, we may conclude that the $q$-integral transformation on this case produces another monomial solution from a monomial solution.

\section{Concluding remarks} \label{sec:CR}
In this paper, we found kernel function identities for the $q$-Heun equation and the variants of the $q$-Heun equation of degree three and four.
We applied them to obtain $q$-integral transformations of solutions to the $q$-Heun equation and its variants.
Moreover, we investigated special solutions of the $q$-Heun equation from the perspective of the $q$-integral transformation.

We give some comments on issues related to the results in this paper.

The variants of the $q$-Heun equation was found by considering degenerations of the Ruijse\-naars--van Diejen system, and kernel functions of the Ruijse\-naars--van Diejen system had been found in \cite{KNS,Rui06,Rui09}.
On the other hand, we found kernel functions of the $q$-Heun equation and its variants directly in this paper.
It is expected to obtain degenerations of kernel functions from the non-degenerate Ruijsenaars--van Diejen system, although calculation would be extremely complicated without nice ideas.
Extension of our identities of kernel functions to the multivariable cases is also expected.

We investigated $q$-integral transformations of monomial-type solutions of the $q$-Heun equation in Section~\ref{sec:solqHeun}.
Polynomial solutions of the $q$-Heun equation and its variants were investigated in \cite{KST,TakqH}, and it is expected to study $q$-integral transformations related to the polynomial-type solutions.
Results in \cite{TakIT} for Heun's differential equation might be helpful on this direction.
Suitable analogues of Corollary \ref{cor:A4} to integral transformations of the variants of $q$-Heun equations (Theorems \ref{thm:A3} and \ref{thm:A2}) might be useful for further study.

In \cite{STT2}, the $q$-middle convolution by Sakai and Yamaguchi \cite{SY} was applied to the linear $q$-difference equation related with the $q$-Painlev\'e equation whose symmetry is of type \smash{$D^{(1)}_5$}, and the $q$-integral transformation of the $q$-Heun equation was obtained as a corollary.
The variants of the $q$-Heun equation of degree three and four are related with the $q$-Painlev\'e equations of~type \smash{$E^{(1)}_6$} and \smash{$E^{(1)}_7$} through the space of initial conditions \cite{STT1}.
In this paper, we obtained $q$-integral transformations of solutions to the $q$-Heun equation and its variants by using the identities of the kernel functions.
It seems that relationship between the variants of the $q$-Heun equation of degree three and four with the $q$-Painlev\'e equation of type \smash{$E^{(1)}_6$} and \smash{$E^{(1)}_7$} through the $q$-middle convolution is not fully understood.
The paper \cite{FN} by Fujii and Nobukawa might be related to this problem.

\subsection*{Acknowledgements}

The author is grateful to Yumi Arai for discussion.
He also thanks to the referees for valuable comments.
He is supported by JSPS KAKENHI Grant Number JP22K03368.

\pdfbookmark[1]{References}{ref}
\LastPageEnding

\end{document}